\documentclass[12pt]{article}
\usepackage{amsmath,amsthm,amsfonts,amssymb}
\usepackage{graphicx, epsfig, graphics}

\usepackage[usenames]{color} 
\usepackage{pdfsync}

\usepackage{enumerate}

\usepackage{pictex,dcpic}

\usepackage{ulem}
\usepackage{mathrsfs}

\topmargin 0in \newenvironment{pf}{\begin{proof}}{\end{proof}}

\newtheorem{thm}{Theorem}[section]

\newtheorem{theorem}{Theorem}[section]

\newtheorem{lemma}[thm]{Lemma}

\newtheorem{corollary}[thm]{Corollary}
\newtheorem{proposition}[thm]{Proposition}
\newtheorem{questions}[thm]{Questions}

\newtheorem*{thm8.1}{Theorem~8.1}

\newtheorem*{sublem*}{Sublemma}
\theoremstyle{definition}

\newtheorem{definition}[thm]{Definition} 

\newtheorem{notn}[thm]{Notation}

\newtheorem{ex}[thm]{Example}

\newtheorem{remark}[thm]{Remark}
\theoremstyle{remark}

\newcounter{remarks}

\newcounter{enumitemp}
\newenvironment{enumeratecontinue}{
\setcounter{enumitemp}{\value{enumi}}
\begin{enumerate}
\setcounter{enumi}{\value{enumitemp}}
}
{
\end{enumerate}
}

\newcommand\pref[1]{(\ref{#1})}

\DeclareMathOperator{\Fix}{\mathsf{Fix}}

\newcommand{\R}{{\mathbb R}}
\newcommand{\Z}{{\mathbb Z}}

\newcommand{\f}{{F_n}}

\newcommand{\Out}{\mathsf{Out}}
\newcommand{\Aut}{\mathsf{Aut}}

\newcommand{\Hom}{\mathsf{Hom}}

\newcommand{\Stab}{\mathsf{Stab}}

\newcommand{\upg}{UPG}
\newcommand{\vf}{\mathsf{VF}}

\newcommand{\ti} {\tilde}

\newcommand{\cR}{\mathcal R}

\def\<{\prec}
\def\bN{\mathbb N}
\def\cC{\mathcal C}
\def\cH{\mathcal H}
\def\sgs{\mathcal K}	
\def\sg{K}	
\def\st{St}
\def\Krstic{Krsti\'{c}}
 
\def\gph{\Gamma}

\usepackage[usenames]{color} 

\usepackage[svgnames]{xcolor}
\usepackage
{setspace}

\DeclareMathOperator{\Ker}{Ker}

\def\bX{\mathbf X}
\def\bY{\mathbf Y}
\def\bZ{\mathbf Z}
\def\bW{\mathbf W}
\def\sfY{{\mathsf Y}}

 \def\Y{Y}
 \def\i{m} 
 \def\Rn{R_n}
 \def\group{{\mathsf G}}

\def\Y{{\mathsf Y}}
\def\mg{{\tau}}
\def\Gmg{{\sigma}}

\def\pf{presentation family}
\def\gmc{{GMc}}

\author{Mladen Bestvina\thanks{M.B.~gratefully acknowledges the support by the National Science Foundation
under the grant number DMS-1905720}, Mark Feighn and Michael Handel}
\title{A McCool Whitehead type theorem for finitely generated subgroups of $\Out(F_n)$}
\date{\today}

\begin{document}
\maketitle

\begin{abstract}
S.~Gersten announced an algorithm that takes as input two finite sequences $\vec K=(K_1,\dots, K_N)$ and $\vec K'=(K_1',\dots, K_N')$ of conjugacy classes of finitely generated subgroups of $\f$ and outputs:
\begin{enumerate}
\item
 {\tt YES} or {\tt NO} depending on whether or not there is an element $\theta\in \Out(\f)$ such that $\theta(\vec K)=\vec K'$ together with one such $\theta$ if it exists and
\item
a finite presentation for the subgroup of $\Out(\f)$ fixing $\vec K$. 
\end{enumerate}
S.~Kalajd{\v{z}}ievski published a verification of this algorithm. We present a different algorithm from the point of view of Culler-Vogtmann's Outer space.

New results include that the subgroup of $\Out(\f)$ fixing $\vec K$ is of type $\vf$, an equivariant version of these results, an application, and a unified approach to such questions. 
\end{abstract}

\tableofcontents

\section{Introduction} \label{s:Introduction}

Let $\sgs$ be the set of conjugacy classes of finitely generated subgroups of $\f$  and let $X_\sgs$ be the set of finite sequences of elements of $\sgs$.  For our first main result, we consider the action of  $\Out(F_n)$ on  $X_\sgs$. For $\vec K:=(K_1,\dots,, K_N)\in X_\sgs$, the subgroup $\gmc(\vec K)$ of $\Out(\f)$ fixing $\vec K$ in $X_\sgs$ is called a {\it generalized McCool group of $\f$}. See Section~\ref{s:compare} for more on terminology.

\begin{theorem}\label{T1}
There is an algorithm that takes as input two finite sequences $\vec K=(K_1,\dots, K_N)$ and $\vec K'=(K_1',\dots, K_N')$ of conjugacy classes of finitely generated subgroups of $\f$ and outputs:
\begin{enumerate}
\item \label{item:T1W}
 {\tt YES} or {\tt NO} depending on whether or not there is an element $\theta\in \Out(\f)$ such that $\theta(\vec K)=\vec K'$ together with one such $\theta$ if it exists and
\item \label{item:T1M}
  a finite presentation for $\gmc(\vec K)$.
\end{enumerate}
Further,
\begin{enumeratecontinue}
\item\label{item:T1VFL}
$\gmc(\vec K)$ is of type $\vf$.
\end{enumeratecontinue}
\end{theorem}

If one replaces $\sgs$ with the set $\cC$ of   conjugacy classes of elements of $\f$,  then this theorem is classical: \pref{item:T1W}  is due to  J.H.C.~Whitehead \cite{jhcw:equivalent, jhcw:certain} (who also computes a finite generating set for the stabilizer of the elements);  \pref{item:T1M}  is due to McCool  \cite{jm:fp}; and \pref{item:T1VFL} is due to Culler and Vogtmann \cite{cv:moduli}.

Items \pref{item:T1W} and \pref{item:T1M} of Theorem~\ref{T1} were announced by Gersten \cite{sg:whitehead} and proved by Kalajd{\v{z}}ievski \cite{sk:gersten}. Gersten's result plays an important role in the solution \cite{fh:upg} of the conjugacy problem for UPG elements of $\Out(\f)$. A motivation for this paper was an accessible proof of Gersten's result from the point of view of Culler-Vogtmann's Outer space.

The subgroup of $\Out(G)$ fixing a given finite sequence of conjugacy classes of elements in $\f$ has been termed an {\it elementary McCool group of $\f$} by Guirardel-Levitt \cite{gl:mccool}. See also Section~\ref{s:compare}. They have announced \cite{gl:outboundary} a structure theorem for the stabilizers of points in the boundary of Outer space where the key terms are elementary McCool groups of $\f$. In particular, they show that these stabilizers are of type $\vf$, that is, some finite index subgroup has a finite Eilenberg-MacLane space. Note that, up to finite index, an elementary McCool group is  a generalized McCool group.

For a group $\group$,   $\Out(\f)$ acts on the set $\Hom\big(\group,\Out(\f)\big)$ of representations via
$\theta(\alpha)=i_\theta\circ \alpha$ {where $i_\theta$ denotes conjugation by $\theta$.   For our second main result, we consider the action of $\Out(\f)$
  on the product of $\Hom\big(\group,\Out(\f)\big)$ and $X_\sgs$, obtaining the following generalization of Theorem~\ref{T1}. 
  
\begin{theorem}[Generalized Theorem~\ref{T1}] \label{T2}
For any finite group $\group$, there is an algorithm that takes as input two homomorphisms $\alpha,  \alpha' : \group\to  \Out(\f)$ and two finite sequences   $\vec K=(K_1,\dots, K_N)$ and $\vec K'=(K_1',\dots, K_N')$ of conjugacy classes of finitely generated subgroups of $\f$ and outputs:
\begin{enumerate}
\item \label{item:T2W}
 {\tt YES} or {\tt NO} depending on whether or not there is an element $\theta\in \Out(\f)$ such that  $\theta^{-1}\alpha(g)\theta=\alpha'(g)$ for all $g \in \group$ and such that $\theta(\vec K)=\vec K'$. Further the algorithm finds one such $\theta$ if it exists.
\item \label{item:T2M}
  a finite presentation for the subgroup of $\Out(\f)$  fixing $\vec K$ and commuting with each element of $\alpha(\group)$.
\end{enumerate}
Further,
\begin{enumeratecontinue}
\item\label{item:T2VFL}
 The  subgroup of $\Out(\f)$  fixing $\vec K$ and commuting with each element of $\alpha(\group)$ is $\vf$. In the terminology of Section~\ref{s:compare}, equivariant generalized McCool groups of $\f$ are of type $\vf$.
 \end{enumeratecontinue}
\end{theorem}

\begin{remark} Even though Theorem~\ref{T2} implies Theorem~\ref{T1}, we include a  proof of the latter to serve as a less technical warm-up for that of Theorem~\ref{T2}.
\end{remark}

If one replaces $\sgs$ with $\cC$,  then   Theorem~\ref{T2}\pref{item:T2W} is due to  \Krstic, Lustig, and Vogtmann \cite[Theorem~1.1]{klv:conjugacy}. They go on in that same paper to extend the solution  to the conjugacy problem for Dehn twists \cite{cl:conjugacy} to  all linear growth elements of $\Out(\f)$. Theorem~\ref{T2} should be useful in extending the solution of the \upg\ conjugacy problem to all polynomially growing elements of $\Out(\f)$. See \ \cite[Theorem~A]{sk:gersten} for some progress toward Theorem~\ref{T2}\pref{item:T1VFL}.

Following  \cite{cv:moduli} and \cite {kv:equivariant}, we associate a subcomplex of the spine of Outer space to each  pair $(\vec K,\alpha)$ (or just $\vec K$ in the context of Theorem~\ref{T1}) and prove that this subcomplex is contractible.   We follow the proofs in \cite{kv:contractible} (which revisits \cite{cv:moduli})   and    \cite{kv:equivariant}  closely.  In those papers, $\sgs$ is replaced by $\cC$ and the construction of the subcomplex is based upon the length $\|\tau\|_w$ of the realization of $w \in \cC$ in a marked graph $\tau$.  We replace this with the volume   $\|\tau\|_K$ of the Stallings graph determined by  $K \in \sgs$ and a marked graph $\tau$; see Section~\ref{s:Stallings}.  Most of the arguments in the proof of contractibility are unaffected by this change and we do not  repeat their proofs here.   The elements of the proofs that are affected are these. In the non-equivariant case  ($\group = \{1\}$),  i.e. Theorem~\ref{T1}, {\cite[Lemma~11.2]{kv:contractible}} must be replaced and we do this in Lemma~\ref{l:11.2}.  In the equivariant case,   i.e. Theorem~\ref{T2},  we replace \cite[Lemmas 7.1 and 7.2]{kv:equivariant} with    Lemmas~\ref{l:7.2} and \ref{l:7.1} respectively.   This is the most technical part of the paper and is carried out in Section~\ref{s:combinatorial}.

With contractibility in hand, we deduce all three items in our main theorems at once in a simple, but seemingly new, way.  For example,  we do not use peak reduction. The keys here are Definition~\ref{defn:gfp} and Proposition~\ref{p:gfp}.  See also {Proposition~\ref{presentation family 1}.

We end the paper with an application. Suppose $K$ is the conjugacy class of a finitely generated subgroup of $\f$ that is its own normalizer. There is then a natural restriction map from $\gmc(K)$ to $\Out(K)$. The image $Ext(\f;K)$ of this map is the subgroup of $\Out(K)$ of elements that extend to $\Out(\f)$. 
\begin{thm8.1}[special case]
With $K$ as above, $Ext(F_n;K)$ is commensurable to a finite product of
  generalized McCool groups. In particular, it is of type $\vf$.
\end{thm8.1}

We thank the referee for a careful reading and helpful comments.

\section{McCool Groups and Generalized McCool Groups}\label{s:compare}
In this section, added at the suggestion of the referee and motivated by \cite{gl:mccool}, we slightly generalize Theorem~\ref{T2} in Theorem~\ref{T3}. A reader interested only in the proofs of Theorems~\ref{T1} and \ref{T2} can skip this section. Some terminology is needed.

\begin{definition}[{\cite[Section~2]{gl:mccool}}]
Let $G$ be a group and $H$ a subgroup. Say that $\phi\in\Out(G)$ {\it acts trivially on}, respectively {\it preserves}, the conjugacy class $[H]$ of $H$ in $G$ if there is a representative $\Phi\in\Aut(G)$ for $\phi$ such that $\Phi$ is the identity on $H$, respectively such that $\Phi(H)=H$. Note that if $\phi$ acts trivially on $[H]$ then $\phi$ preserves $[H]$, but not necessarily vice versa. 
\end{definition}

\begin{definition}[cf.~{\cite[Definition~2.1]{gl:mccool}}]\label{d:mccool}
Let $G$ and $\group$ be groups, $\alpha:\group\to \Out(G)$ a homomorphism, and $\vec H=\{[H_i]:i\in I\}$ and $\vec K=\{[K_j]:j\in J \}$ indexed sets of conjugacy classes of subgroups of $G$. We denote by $\Out(G; \vec H^{(t)}, \vec K, \alpha)$ the subgroup of $\Out(G)$ consisting of outer automorphisms $\phi$ that:
\begin{itemize}
\item
act trivially on each element of $\vec H$;
\item
preserve each element of $\vec K$; and
\item
commute with each element of $\alpha(\group)$.
\end{itemize}
If $\group$ is trivial, we often write $\Out(G;\vec H^{(t)}, \vec K)$ instead of $\Out(G; \vec H^{(t)}, \vec K, \alpha)$. If $I$ is finite, then we call $\vec H$ a finite sequence.

\end{definition}

\begin{theorem}[{\cite[first part of Theorem~3.4]{gl:mccool}}]\label{t:gl}
If $G$ is a toral relatively hyperbolic group and $\vec H$, $\vec K$ are finite sequences of finitely generated subgroups, with every element of $\vec K$ abelian, then $\Out(G;\vec H^{(t)}, \vec K)$ is of type $\vf$.
\end{theorem}

\begin{theorem}[Generalized Theorem~\ref{T2}] \label{T3}
There is an algorithm that takes as input:
\begin{itemize}
\item
a finite group $\group$ and two homomorphisms $\alpha,  \alpha' : \group\to  \Out(\f)$;
\item
two finite sequences $\vec H$, $\vec H'$ of conjugacy classes of finitely generated subgroups of $\f$; and
\item
two finite sequences $\vec K$, $\vec K'$ of conjugacy classes of finitely generated subgroups of $\f$
\end{itemize}
and outputs:
\begin{enumerate}
\item \label{item:T3W}
 {\tt YES} or {\tt NO} depending on whether or not there is an element $\theta\in \Out(\f)$ such that
 \begin{itemize}
 \item
 $\theta^{-1}\alpha(g)\theta=\alpha'(g)$ for all $g \in \group$;
 \item
 $\theta(\vec H)=\vec H'$; and
 \item
 $\theta(\vec K)=\vec K'$. 
 \end{itemize}
 Further the algorithm finds one such $\theta$ if it exists.
\item \label{item:T3M}
  a finite presentation for $\Out(\f;\vec H^{(t)},\vec K,\alpha)$.
\end{enumerate}
Further,
\begin{enumeratecontinue}
\item\label{item:T3VFL}
 $\Out(\f;\vec H^{(t)},\vec K,\alpha)$ is of type $\vf$.
 \end{enumeratecontinue}
\end{theorem}

After a lemma, we prove the theorem.

\begin{lemma}\label{l:T3}
There is an algorithm with input a finite sequence  $$\vec H=([H_1], [H_2], \dots, [H_k])$$ of conjugacy classes of subgroups of $\f$ and output a finite set $\{[g_i]\}$ of conjugacy classes of elements of $\f$ such that $\phi\in\Out(\f)$ acts trivially on each $[H_j]$ iff $\phi$ fixes each $[g_i]$.
\end{lemma}

\begin{proof}
Suppose we have a finite sequence $\vec C$ of elements of conjugacy classes of elements of $\cup_j H_j$. Use Whitehead \cite{jhcw:equivalent, jhcw:certain} or McCool \cite{jm:fp}  to compute a finite generating set for the stabilizer of $\vec C$. Check whether each generator acts trivially on each $H_j$. If not, concatenate to $\vec C$ an element that is moved. In this way we get a strictly decreasing chain of elementary McCool groups of $\f$. By  
\cite[Theorem~1.5 and Corollary~1.6]{gl:mccool} such a chain is bounded and so this process terminates.
\end{proof}

\begin{remark}
\cite[Corollary~1.6]{gl:mccool} shows that such a finite set $\{g_i\}$ exists (even if $\vec H$ is an arbitrary set of not necessarily finitely generated subgroups and in the setting of toral relatively hyperbolic groups). The point of  Lemma~\ref{l:T3} is that, in its setting, finding such elements can easily be made algorithmic.
\end{remark}

\begin{proof}[Proof of Theorem~\ref{T3}]
\pref{item:T3W} follows by applying Theorem~\ref{T2}\pref{item:T2W} to the concatenation $\vec H \vec K$ of $\vec H$ and $\vec K$.

(\ref{item:T3M}, \ref{item:T3VFL}): Apply the algorithm of Lemma~\ref{l:T3} to $\vec H$ to obtain the finite set $\{[g_i]\}$. Let $\vec H_0:=([\langle g_1\rangle], [\langle g_2\rangle], \dots)$ be the corresponding finite sequence of conjugacy classes of cyclic subgroups of $\f$. $\hat{\mathsf O}:=\Out(\f;\emptyset, \vec H_0^{}\vec K,\alpha)$ acts on $\{[g_i], [g_i^{-1}]\}$ with kernel $\mathsf O:=\Out(\f;\vec H^{(t)},\vec K,\alpha)$. By Theorem~\ref{T2}\pref{item:T2VFL}, $\hat{\mathsf O}$, and hence also $\mathsf O$, has type $\vf$ thus proving \pref{item:T3VFL}. Using the finite presentation of $\hat{\mathsf O}$ provided by Theorem~\ref{T2}\pref{item:T2M}, we can compute the the image of $\hat{\mathsf O}$ in the permutation group of $\{[g_i], [g_i^{-1}]\}$ and then use the Reidemeister-Schreier theorem (see for example the section on finite presentations and finite index subgroups of \cite{fh:upg} in this setting), to finitely present $\mathsf O$.
\end{proof}

One could ask about generalizations of Theorem~\ref{T3}.

\begin{questions}
Under what conditions is $\Out(\f;\vec H^{(t)};\vec K;\alpha)$ of type $\vf$? Is there a sequence $\vec K=\{[K_j]:j\in \mathbb N\}$ of conjugacy classes of finitely generated subgroups of $\f$ such that $\Out(\f;\emptyset,\vec K)$ does not have type $\vf$? Does $\Out\big(\f;\emptyset, ([K_1])\big)>\Out\big(\f;\emptyset, ([K_1], [K_2])\big)>\dots$ always stabilize?
\end{questions}

\begin{remark}
\begin{itemize}
\item
Using \cite[Corollary~1.6]{gl:mccool}, Theorem~\ref{T3}\pref{item:T3VFL} remains true if the condition on $\vec H$ is replaced by the assumption that $\vec H$ is an arbitrary indexed set of conjugacy classes of (not necessarily finitely generated) subgroups of $\f$. 
\item
\def\torelli{\mathcal T_n}
For $n\ge 3$, there is a finite sequence $\vec K$ of conjugacy classes of (infinitely generated) subgroups of $\f$ so that $\Out(\f;\emptyset, \vec K)$ does not have type $\vf$. Indeed, the Torelli group $\torelli$ of $\Out(\f)$ is torsion free and does not have type $\mathsf F$ \cite{bbm:torelli} (for $n=3$, this is Krsti\'c and McCool \cite{km:torelli} and, for $3\leq n<100$, this is Smillie-Vogtmann \cite{sv:euler}). $\torelli$ is the kernel of the action of $\Out(\f)$ on $H_1(\f;\Z)$ and so equivalently on the abelian group of homomorphisms $\f\to\Z$, which is isomorphic to $\Z^n$. If $f$ is a homomorphism from $\f$ to $\Z$ and $\phi\in\Out(\f)$ preserves its kernel, then there is an induced $\overline\phi\in\Out(\Z)$ such that $f\circ\phi=\overline \phi\circ f$, and so $f\circ\phi=\pm f$. It follows that if $\{f_1, f_2,\dots, f_n\}$ is a basis then $\torelli$ has index $2^n$ in $\Out\big(\f;\emptyset, (\Ker(f_1), \Ker(f_2),\dots)\big)$. Hence the latter group does not have type $\vf$.
\end{itemize}
\end{remark}



We end this section with a bit more terminology.

\begin{definition}[{cf.~\cite[Definitions~1.1, 1.2, and 2.1]{gl:mccool}}]\label{d:mccool2}
Let $G$, $\group$, $\vec H$, $\vec K$, and $\alpha$ be as in Definition ~\ref{d:mccool}. Additionally, let $\vec C$ denote a set conjugacy classes of elements of $G$.
\begin{itemize}
\item
$Mc_G(\vec C)$ denotes the subgroup of $\Out(G)$ fixing each $[c]\in\vec C$. Equivalently, $Mc_G(\vec C)=\Out(G;\vec H^{(t)},\emptyset)$ where $\vec H=\{[\langle c\rangle] \mid [c]\in\vec C)\}$. If $\vec C$ is finite, then we say that $Mc_G(\vec C)$ is an {\it elementary McCool group of $G$ (or of $\Out(G)$)}.
\item
$Mc_G(\vec H):=\Out(G;\vec H^{(t)},\emptyset)$. If $\vec H$ is finite with every element finitely generated then we say that $Mc_G(\vec H)$ is a {\it McCool group of $G$ (or of $\Out(G)$)}.
\item
$\gmc_G(\vec K):=\Out(G;\emptyset,\vec K)$. If $\vec K$ is finite with every element finitely generated then we say that $\gmc_G(\vec H)$ is a {\it generalized McCool group of $G$ (or of $\Out(G)$)}.
\item
If $\group$ is not trivial, then {\it equivariant elementary, equivariant, or equivariant generalized McCool groups} are defined analogously.
\item
If $G=\f$ then we usually suppress the subscript.
\end{itemize}
\end{definition}

\begin{remark}
\begin{itemize}
\item
Guirardel-Levitt \cite[Corollary~1.6]{gl:mccool} show that if $G$ is toral relatively hyperbolic then $Mc(\vec H)$ is a McCool group and also that any McCool group is an elementary McCool group.
\item
$Mc_G(\vec K)$ is a normal subgroup of $\gmc_G(\vec K)$. In case $G=\f$, Theorem~\ref{ext} shows that the quotient $Ext(\f;\vec K)$ is of type $\vf$.
\item
There are generalized McCool groups of $\f$ that are not commensurable in $\Out(\f)$ to a McCool group of $\f$. For example, there are $\phi\in\Out(\f)$, $n\ge 4$, preserving a free factor $F$ but without any non-trivial periodic conjugacy classes of non-trivial elements of $\f$. In this case, $\gmc([F])$ is such. Indeed, an iterate of $\phi$ appears in a given finite index subgroup of $\gmc([F])$. Since an iterate doesn't fix any non-trivial conjugacy classes, it is not contained in an elementary McCool group of $\f$. Using the first bulleted item of this remark, an iterate is therefore not contained in a McCool group of $\f$.
\end{itemize}
\end{remark}

\begin{remark}
The notations $Mc_G(\vec C)$ and $Mc_G(\vec H)$ and the definitions {\it elementary McCool group} and {\it McCool group} correspond with those in \cite{gl:mccool} with the small differences that Guirardel and Levitt use sets and we use indexed sets and that they use $\cC, \cH, \sgs$ for our $\vec C, \vec H, \vec K$ and we use $\cC$ (respectively $\sgs$) for the set of conjugacy classes of elements (respectively finitely generated subgroups) of $\f$.
\end{remark}

\begin{remark}
Except in this section, we only consider finite $\vec K$ with every element finitely generated, i.e.\ $\gmc(\vec K)$ will be a generalized McCool group. The notation $Mc(\cdot)$ and $\gmc(\cdot)$ and the terminology {\it elementary McCool, McCool}, and {\it generalized McCool group} don't return until the last section, Section~\ref{s:ext}.
\end{remark}

\section{Background} \label{s:background}
Most arguments in this paper take place in the spine of Outer space, a space introduced by Culler and Vogtmann in their seminal paper \cite{cv:moduli}. In this section we quickly recall key definitions and notation needed for these arguments. See also \cite{kv:contractible, kv:equivariant}. 

$\Rn$ denotes a specific rose, i.e.\ a graph with one vertex and $n$ edges, whose fundamental group has been identified with $\f$. A {\it marked graph} is a pair $(G,g)$ consisting of a graph $G$ and a homotopy equivalence $g:\Rn\to G$. $G$ is required to have no vertices of valence $\le 2$. Marked graphs $\mg=(G,g)$ and $\mg'=(G',g')$ are {\it equivalent} if there is cellular isomorphism\footnote{Recall that a map between $CW$-complexes is {\it cellular} if it takes $k$-skeleta to $k$-skeleta. It is a {\it cellular isomorphism} if it is a cellular homeomorphism with a cellular inverse.} $h:G\to G'$ such that $h\circ g$ and $g'$ are homotopy equivalent. We write $[(G,g)]$ for the equivalence class, but often blur the distinction between a marked graph and its equivalence class by omitting the brackets.

A homotopy equivalence is a {\it (forest) collapse} if it is of the form $c:G\to G/F$ where $F\subset G$ is a forest and $G/F$ is the quotient space obtained by collapsing components of $F$. We also say that $(G/F, c\circ g)$ is obtained from $(G,g)$ by a {\it forest collapse}. The {\it spine} of Outer space}   is defined to be the geometric realization of the partial order induced by forest collapse on the set of equivalence classes of marked graphs. $\Out(\f)$ acts on the spine: if $h:\Rn\to\Rn$ realizes $\theta\in\Out(\f)$ then $(G,g)\cdot \theta:=(G,g\circ h)$. This is a right action. We will have one occasion (in the proof of Proposition~\ref{presentation family 1}) to use the left action defined as $\theta\cdot (G,g):=(G,g)\cdot \theta^{-1}$.

$L_n$ denotes the $\Out(\f)$-subcomplex of the spine spanned by graphs with no separating edge. We usually suppress the subscripts $n$. $\cR$ denotes the set of equivalence classes of marked roses. Since every marked graph collapses to a rose, $L$ is the union of the stars of the elements of $\cR$, i.e. $L=\cup \{St(\rho)\mid \rho\in\cR\}$. Note that $\Out(\f)$ acts transitively on $\cR$. Indeed, if $(R,r)\in\cR$ and $h:R\to R_n$ is a cellular isomorphism then $h\circ r:\Rn\to\Rn$ represents some $\theta\in\Out(\f)$ and $(R_n, id)\cdot\theta=(R,r)$.

If $f:R\to R'$ is a homotopy equivalence of roses of the form $R\leftarrow G\to R'$ where each map is an edge collapse  then we say that $\rho':=(R', f\circ r)\in \cR$ is obtained from $\rho:=(R,r)\in\cR$ by a {\it Whitehead move}. Note that $St(\rho)\cap St(\rho')\not=\emptyset$ and that, for any $\theta\in\Out(\f)$, $\rho'\cdot\theta$ is obtained from $\rho\cdot\theta$ by a Whitehead move.

\section{Presentation Families}\label{s:gfp}
The main work in proving Theorems~\ref{T1} and \ref{T2} is proving that certain subcomplexes of the spine of Outer space are contractible.  In this section we describe how to complete the proofs once this has been done.

 In the model case for the following definition, $G =\Out(F_n)$, $X$ is the set $X_\cC$  of finite sequences of conjugacy classes of elements in $\f$, $L$ is the subcomplex of the spine of Outer space spanned by marked graphs without separating edges and, for $x\in X$, $L_x$ is the subcomplex of $L$ that is the union of the stars of the roses in which the length of $x$ is minimized.

\begin{definition}
\label{defn:gfp}
A $G$-set $X$ has a {\it \pf}\ if:
\begin{enumerate}
\item\label{item:properties of G}
there is an algorithm that takes as input $g_1,g_2,g_3\in G$ and outputs {\tt YES} or {\tt NO} depending on whether or not $g_1g_2=g_3$.
\item \label{item:properties of X}
there is an algorithm that takes as input $g\in G$ and $x_1, x_2\in X$ and outputs {\tt YES} or {\tt NO} depending on whether or not $gx_1=x_2$.
\item \label{item:properties of L} there is a  simplicial, properly discontinuous action of $G$ on a simplicial complex $L$ with the following properties.
\begin{enumerate}
\item  \label{item:no inversion}
$G$ acts without inverting edges of $L$.
\item  \label{item:Gvw}
For all vertices, $v, w$ of $L$, there is an algorithm that decides if there exists $g \in G$ such that $gv = w$;  If {\tt YES}  then   the algorithm outputs  the finite set $G_{v,w} =  \{g\mid gv=w\}$.
\end{enumerate}
\item  
For each $x \in X$ there is a contractible $G_x$-invariant subcomplex $L_x$ of $L$ with the following properties where $G_x:=\{g\in G\mid gx=x\}$.
\begin{enumerate}
\item\label{item:domain for Lx}
there is an algorithm that takes as input $x\in X$ and outputs a finite fundamental domain $D_x$  for 
 the $G_x$-action on $L_x$; more precisely,  $D_x$ is a finite subcomplex of $L_x$ such that $G_xD_x=L_x$ and   we are given the  set $V$ of vertices of $D_x$ together with the finite subset of $2^V$ consisting of the simplices of $D_x$.
\item
$\theta(L_x)=L_{\theta(x)}$ for all $\theta \in G$.\label{item:natural}
\end{enumerate}
\end{enumerate}

In this case, we say that   $\{L_{x}\mid x\in X\}$ is a {\it  \pf\  in $L$} for the action of $G$ on $X$.
\end{definition}

Recall that a group has type $\vf$ if some finite index subgroup has a finite Eilenberg-MacLane space. 

\begin{proposition}\label{p:gfp}
Suppose the $G$-set $X$ has a \pf\ and $G$ has a  torsion-free  subgroup of finite index. Then there is an algorithm that takes as input $x,y\in X$ and outputs
\begin{enumerate}
\item \label{item:W}
{\tt YES} or {\tt NO} depending on whether or not there is $g\in G$ such that $gx=y$ together with $g$ if it exists;
\item \label{item:M}
a finite presentation for $G_x:=\{g\in G|gx=x\}$.
\end{enumerate}
Further,
\begin{enumeratecontinue}
\item \label{item:VFL}
$G_x$ is of type $\vf$.
\end{enumeratecontinue}
\end{proposition}

\medspace

\proof  We adopt the notation of Definition~\ref{defn:gfp}. We are assuming that  the action of $G$ on $L$ is properly discontinuous and that  $G$ has a finite index torsion-free subgroup, say $G'$. It follows that the action of $G'$ on $L$, and hence the action of  $G' \cap G_x$ on $L_x$, is free.  Item \pref{item:VFL} of Proposition~\ref{p:gfp} therefore follows from Definition~\ref{defn:gfp}\pref {item:domain for Lx}, which implies that the action of $G' \cap G_x$ on $L_x$ is cocompact. 

We use K.~Brown \cite{kb:presentation} to give a finite presentation for $G_x$ based on its action on $L_x$.  Our actions do not invert edges and so we use the simplification of Brown's presentation given in \cite[Theorem 2]{afv:presentation}.
The only things needed to make this algorithmic are:
\begin{itemize}
\item
an explicit description of a fundamental domain $D_x$ for this action;
\item
for each $v,w \in D_x$, the finite set $\{g \in G_x \mid gv = w\}$; and
\item
for each $v\in D_x$, a presentation for the finite group $\{g\in G_x\mid gv=v\}$.  
\end{itemize}
These are provided, respectively,  by:
\begin{itemize}
\item
Definition~\ref{defn:gfp}\pref{item:domain for Lx};
\item
Definition~\ref{defn:gfp}\pref{item:Gvw} and the ability to determine if a given $g \in G$ is contained in $G_x$ (Definition~\ref{defn:gfp}\pref{item:properties of X});
\item
Definition~\ref{defn:gfp}\pref{item:properties of G}, which allows us to compute its multiplication table (and hence a presentation).
\end{itemize}
This completes the proof of Proposition~\ref{p:gfp}\pref{item:M}.

It remains to show that it is algorithmic to check, given $x,y \in X$, whether there is $g \in G$ such that $g(x) =y$ and to find such a $g$ if it exists. Let $D_x \subset L_x$ and $D_y \subset L_y$ be the domains given in item \pref{item:domain for Lx} of Definition~\ref{defn:gfp}.  Item \pref{item:natural} of Definition~\ref{defn:gfp} implies that if $g(x) = y$ then $g(L_x) = L_y$.    In this case,  (by pre-composing with an element of $G_x$  and post-composing with an element of $G_y$) there is such a $g \in G$ taking a vertex of $D_x$ to one of $D_y$. Using Definition~\ref{defn:gfp}\pref{item:Gvw} we then check, for each $g$ taking a vertex of $D_x$ to a vertex of $D_y$, whether $g$ takes $x$ to $y$. If we find one, then this is the desired $g$; if not then there is no $g$ taking $x$ to $y$.   This completes the proof of item \pref{item:W} of Proposition~\ref{p:gfp}.
\endproof

\section{Stallings Graphs} \label{s:Stallings}

Recall $\cC$ denotes the set of conjugacy classes of elements of $\f$. 
The constructions in \cite{cv:moduli},  \cite {kv:equivariant} and \cite{kv:contractible} are based on the length  function for an element of $\cC$ in a marked graph $\tau$.  In this section we generalize this to a volume function for an element of $\sgs$ in $\tau$ where $\sgs$ is the set of finitely generated subgroups of $\f$.

Each $w \in \cC$ is represented in a marked graph $\mg = (G,g)$ by a unique circuit, which we view as an immersion $ w_{\mg} : S^1\to G$.  Subdivide $S^1$  so that $ w_{\mg}$ maps each edge in $S^1$   to an edge in $G$ and define the {\it length} $ \|\mg\|_{w}$ of $w$ to  be the number of edges in the subdivided $S^1$. Equivalently, $ \|\mg\|_{w}$ is the number of edges of $G$ (counted with multiplicity) crossed by the circuit $ w_{\mg}$.    If $\mg = (G,g)$ and  $\mg' = (G',g')$ represent the same vertex in the spine  of Outer space then there is a cellular isomorphism $h :G \to G'$  such that $ w_{\mg'} = h w_{\mg} :S^1 \to G'$ and so $ \|\mg'\|_{w} =  \|\mg\|_{w}$.  We may therefore view $ \|\mg\|_{w}$ as an invariant of the point in the spine represented by $\mg$.

Fix a rose $\rho = (R,r)$  and let $H$ be the set of half-edges of $R$, equivalently the set of oriented edges of $R$.  For each $w \in \cC$,    divide  $S^1$, thought of as the domain of $w_\rho$,    into {\it pieces} $\Y = \{Y\}$ by snipping each edge at its midpoint.  Each  $Y $  is a topological arc  that is composed  of two half-edges;  using the half-edges,  label the endpoints of $Y$ by elements of $H$.    To construct the {\it star graph for $\rho$ and $w$}  (see \cite[Section~9]{kv:contractible}), start with $2n$ vertices $V$ labeled by the elements of $H$ and then add 
$\|\rho\|_w$  edges, one for each $Y \in \Y$, with endpoints attached to $V$ according to the labeling  on the endpoints of $Y$. 

\medskip

We now modify the definitions of length function and star graph so that they apply to all elements of $\sgs$, in fact all elements of $X_\sgs$, and not just those of rank one.

A {\it Stallings graph (over $G$)} is a map $\i:\Gamma\to G$ such that: 
\begin{itemize}
\item
$\Gamma$ and $G$ are finite graphs with no vertices of valence $\le 2$;
\item
$G$ is connected;
\item
the components $(\Gamma_1,\dots,\Gamma_N)$ of $\Gamma$ are ordered; and
\item
$\i$ takes each open edge of $\Gamma$ homeomorphically to an open edge of $G$ and the induced maps on the stars of vertices are injective. \end{itemize}
We view each edge of $\Gamma$ as {\it labeled} by its image edge in $G$. Stallings graphs $\i:\Gamma\to G$ and $\i':\Gamma'\to G'$ are {\it equivalent} if there is a cellular isomorphism $h:G\to G'$ with a lift $\eta:\Gamma\to\Gamma'$, also a cellular isomorphism, that restricts to order preserving maps $\eta_i:\Gamma_i\to \Gamma'_i$ on components.
\begin{equation*}
\begindc{\commdiag}[500]
\obj(1,0)[20]{$G'$}
\obj(0,1)[01]{$\Gamma$}
\obj(1,1)[21]{$\Gamma'$}
\obj(0,0)[00]{$G$}
\mor{01}{21}{$\eta$}[\atleft,\solidarrow]
\mor{00}{20}{$h$}[\atleft,\solidarrow]
\mor{01}{00}{$\i$}[\atright,\solidarrow]
\mor{21}{20}{$\i'$}[\atleft,\solidarrow]
\enddc
\end{equation*}

\begin{ex} \label{e:inequivalent}
Label the edges of the rose $R_2$ by the letters $a$ and $b$. The circle $\Gamma$ with three edges labeled $aab$ mapping to $R_2$ preserving labels gives a Stallings graph $\i:\Gamma\to R_2$. Similarly the circle $\Gamma'$ labeled $bba$ gives $\i':\Gamma'\to R_2$. Note that $\i$ and $\i'$ are equivalent, but the two-component Stallings graphs given by $\Gamma\sqcup\Gamma$ and $\Gamma\sqcup\Gamma'$ are not equivalent.
\end{ex}

Recall from the Section~\ref{s:Introduction} that $X_\sgs$ denotes the set of finite sequences in $\sgs$.  A {\it marked Stallings graph} is a pair $(\i,\tau)$ where $\i:\Gamma\to G$ is a Stallings graph  and $\tau=(G,g)$ is a marked graph. We say $(\i,\tau)$ {\it represents $\vec{K}=(K_1, K_2,\dots, K_N) \in X_\sgs$ } if $K_i$ is the image of the induced monomorphism $(\i_i)_*:\pi_1(\Gamma_i) \to  \pi_1(G)$ (where $\pi_1(G)$ has been identified with $\f$ via $\tau$). For each $\vec K  \in X_\sgs$  and  marked graph $\mg=(G,g)$, there is a marked Stallings graph $(\vec K_\tau\to G,\tau)$ that represents $\vec K$, constructed as follows. If $p: \ti G \to G$ is the universal cover of $G$ and $T_{A_i}$ is the minimal subtree for a subgroup $A_i$ representing $K_i$ then $(\vec K_\mg)_i$ is the quotient  of $T_{A_i}$ by the action of $A_i$ and $(K_\mg)_i \to G$ is the immersion induced from $p$. We sometimes abuse notation and write $\vec K_\tau$ for the Stallings graph $\vec K_\tau\to G$.

\begin{proposition}\label{p:fix}
Let $\tau=(G,g), \tau'=(G',g')$ be marked graphs representing elements in the spine of Outer space and $\vec K, \vec K'\in X_\sgs$.
\begin{enumerate}
\item \label{i:conj char}
Marked Stallings graphs $(\Gamma\overset{\i}{\to} G,\tau)$ and $(\Gamma'\overset{\i'}{\to} G, \tau)$ represent the same element of $X_\sgs$ iff there is a label-preserving (i.e.\ $\eta$ is a lift of $id_G$) cellular isomorphism $\Gamma\overset{\eta}{\to}\Gamma'$ inducing $\eta_i:\Gamma_i\to \Gamma'_i$.
\item \label{i:Stallings equiv}
If $[\tau]=[\tau']$ then $\vec K_\tau$ and $\vec K_{\tau'}$ are equivalent. The equivalence class of $\vec K_\tau$ is denoted $\vec K_{[\tau]}$.
 
\item \label{i:Stallings phi}
For $\phi\in\Out(\f)$, $\vec K_{[\tau]\phi}=\big(\phi(\vec K)\big)_{[\tau]}$
\item \label{i:fix orbits}
There is $\phi\in \Out(\f)$ such that $[\tau']\cdot \phi=[\tau]$ and $\phi(\vec K)=\vec K'$ if and only $\vec K_{[\tau]}=\vec K'_{[\tau']}$.
\end{enumerate}
\end{proposition}

\begin{proof}
\pref{i:conj char}
$\implies$: Add forests to $\Gamma_i$ and $\Gamma'_i$ to create covering spaces $\hat\Gamma_i$ and $\hat\Gamma'_i$ of $G$ and $G'$ respectively. Since $\Gamma_i$ and $\Gamma'_i$ have the same images in $\pi_1(G)$, by covering space theory there is a lift $\hat \eta_i:\hat\Gamma_i\to\hat\Gamma'_i$ of $id_G$. Set $\eta_i$ to be the induced map $\Gamma_i\to\Gamma'_i$.

$\Longleftarrow$: $\i=\i'\circ \eta$ and so $\i$ and $\i'$ have the same images in $\pi_1(G)$.

\pref{i:Stallings equiv}: If $[\tau]=[\tau']$, then by definition there is a cellular isomorphism $h:G\to G'$ so that $\tau'\simeq h\circ \tau$. Again by covering space theory there is the desired lift $\vec K_\tau\to \vec K_{\tau'}$.

\pref{i:Stallings phi}: If $f:\Rn\to\Rn$ induces $\phi$ and if $(\i,\tau)$ represents $\phi(\vec K)$ then $\big(\i,(G,g\circ f)\big)$ represents $\vec K$. 

\pref{i:fix orbits}:  $\implies$: Using \pref{i:Stallings phi}, $\vec K'_{[\tau']}=(\phi^{-1}\vec K')_{[\tau'\phi]}=\vec K_{[\tau]}$.

$\Longleftarrow$: Since $\vec K_{[\tau]}=\vec K'_{[\tau']}$, we have the right hand commuting square in the diagram below where $h$ and $\eta$ are cellular isomorphisms. Further, we may choose $f$ so that the left rectangle homotopy commutes.
\begin{equation*}
\begindc{\commdiag}[500]
\obj(0,1)[01]{$\Rn$}
\obj(2,0)[40]{$G$}
\obj(2,1)[41]{$G'$}
\obj(3,0)[60]{$\vec K_\tau$}
\obj(3,1)[61]{$\vec K'_{\tau'}$}
\obj(0,0)[00]{$\Rn$}
\mor{01}{41}{$g'$}[\atleft,\solidarrow]
\mor{61}{41}{$\i'$}[\atright,\solidarrow]
\mor{00}{01}{$f$}[\atleft,\solidarrow]
\mor{00}{40}{$g$}[\atleft,\solidarrow]
\mor{60}{40}{$\i$}[\atright,\solidarrow]
\mor{41}{40}{$h$}[\atright,\solidarrow]
\mor{61}{60}{$\eta$}[\atleft,\solidarrow]
\enddc
\end{equation*}
Define $\phi\in\Out(\f)$ to be induced by $f$. In particular, $[\tau']\phi=[\tau]$. The marked Stallings graph $(m,\tau)$ represents $\vec K$ and $(\i',g'\circ f)$ represents $\phi^{-1}\vec K'$. Since the images in $\pi_1(\Rn)$ of the maps induced by the homotopy equivalences $f^{-1}\circ g'^{-1}\circ \i'$ and $g^{-1}\circ \i$ are equal, $\phi(\vec K)=\vec K'$.\end{proof}

The {\it volume} of  a connected graph $\Gamma$ is the cardinality of its edge set. The {\it volume} of a Stallings graph $\Gamma\to G$ is the sequence of volumes of the components of $\Gamma$. Equivalent graphs have equal volumes. By Proposition~\ref{p:fix}\pref{i:Stallings equiv}, the volume $\|\tau\|_{\vec K}$ of $\vec K_\tau$  depends only on $[\tau]$. When $\vec K$ has only one component, say $\vec K=(K)$, we sometimes view $\vec K$ as an element of $\sgs$ and write for example $\|\tau\|_K:=\|\tau\|_{\vec K}$. Example~\ref{e:inequivalent} shows that the map $\vec K_{[\tau]}\mapsto \big((K_1)_{[\tau]}, \dots, (K_N)_{[\tau]}\big)$ from equivalence classes of Stallings graphs representing elements of $X_\sgs$ to sequences of equivalence classes representing elements of $\sgs$ is not injective. However, $\|\tau\|_{\vec K}=(\|\tau\|_{K_1}, \dots, \|\tau\|_{K_N})$.
By Proposition~\ref{p:fix}\pref{i:Stallings phi}, we have:

\begin{remark}\label{r:fix}
$\|\mg\cdot \phi\|_{\vec K} = \|\mg\|_{\phi^{}(\vec K)}$ for all marked graphs $\tau$, all $\vec K \in X_\sgs$ and all $\phi \in \Out(F_n)$. 
\end{remark}

We now define the modified star graph. Fix a rose $\rho = (R,r)$ and a $K \in \sgs$. Divide $\sg_\rho$   into {\it pieces} $\Y = \{Y\}$ by snipping each edge at its midpoint.  Denote the set of valence one vertices of $Y$ by $\partial Y$.   Each piece $Y$ is a tree that is naturally the union of distinct half-edges meeting at a common vertex  of $\sg_\rho$.   We use those half-edges to label the elements of $\partial Y$ by distinct elements of $H$. 
To construct the modified star graph of $\rho$ and $K$, start with $2n$ vertices $V$ labeled by the elements of $H$ and then add   each $Y \in \Y$, with elements of $\partial Y$   attached to $H$ according to their labels. 
It is clear that if $a\in\f$, $K$ is the conjugacy class of the cyclic subgroup $\langle a \rangle$, and $w$ is the conjugacy class of $a$ then the modified star graph for $\rho$ and $K$ is the same as the star graph for $\rho$ and $w$. 

\medskip

We conclude this section by showing that some standard facts about $\|\mg\|_w$ and star graphs also hold for $\|\mg\|_K$ and modified star graphs.
   
Suppose that $\rho = (R,r) $  is a rose and that $ \mg = (G,g)\in St(\rho)$, i.e.\ there is a marking-preserving map $G\to R$ that collapses a maximal tree $T$.  There is a resulting collapse map $K_{\mg} \to K_\rho$. The edges of $G$ that are  collapsed by $G\to R$, i.e.\ the edges of $T$, are {\it new} edges and the others are {\it old}. Similarly, $K_\rho$ consists of old edges and $K_{\mg} \to K_\rho$ collapses new edges. We have a commuting diagram  
\begin{equation*}
\begindc{\commdiag}[500]
\obj(1,0)[20]{$R$}
\obj(0,1)[01]{$K_{\mg}$}
\obj(1,1)[21]{$K_\rho$}
\obj(0,0)[00]{$G$}
\mor{01}{21}{}[\atleft,\solidarrow]
\mor{00}{20}{}[\atleft,\solidarrow]
\mor{01}{00}{}[\atright,\solidarrow]
\mor{21}{20}{}[\atright,\solidarrow]
\enddc
\end{equation*}
If $\{e_i\}$ is the set of new edges in $G$  and $|e_i|_{\mg,K}$ is the number of edges in $K_\mg$ labeled by either $e_i$ or $\bar e_i$ then $$\|\mg\|_K = \|\rho\|_K + \sum_i |e_i|_{\mg,K}$$

 Following \cite[Definition 8.1]{kv:contractible}, we define  a {\it non-trivial ideal edge} $\alpha$ to be  a partition of $H$ into two subsets of cardinality $\ge 2$ that separate some pair $\{e,\bar e\} \subset H$. Each such $\alpha$ determines a marked graph $\mg \in St(\rho)$ and a collapse map $\mg \to \rho$ that collapses a single edge, say $A$.  There is an induced  collapse map $K_{\mg} \to K_\rho$.      If we snip the old edges of $K_{\mg}$ at their midpoints, the resulting components are indexed by the pieces $Y \in \Y$ of $K_\rho$. If $\ti Y$ is the component corresponding to $Y$, then the  map $\ti Y \to Y$  either is a cellular isomorphism or collapses a single edge labeled $A$.  The latter occurs if and only if $\alpha$ separates $\partial Y$ (thought of as a subset of $H$). 
  
   Motivated by this last  observation, we define $|A|_Z$, for arbitrary subsets $A,Z$ of $H$, to 
be  $1$ if $\{A,A^c\}$ separates $Z$ and $0$ otherwise.  See Figure~\ref{f:ideal}.

\begin{figure}[h!] 
\centering
\includegraphics[width=.4\textwidth]{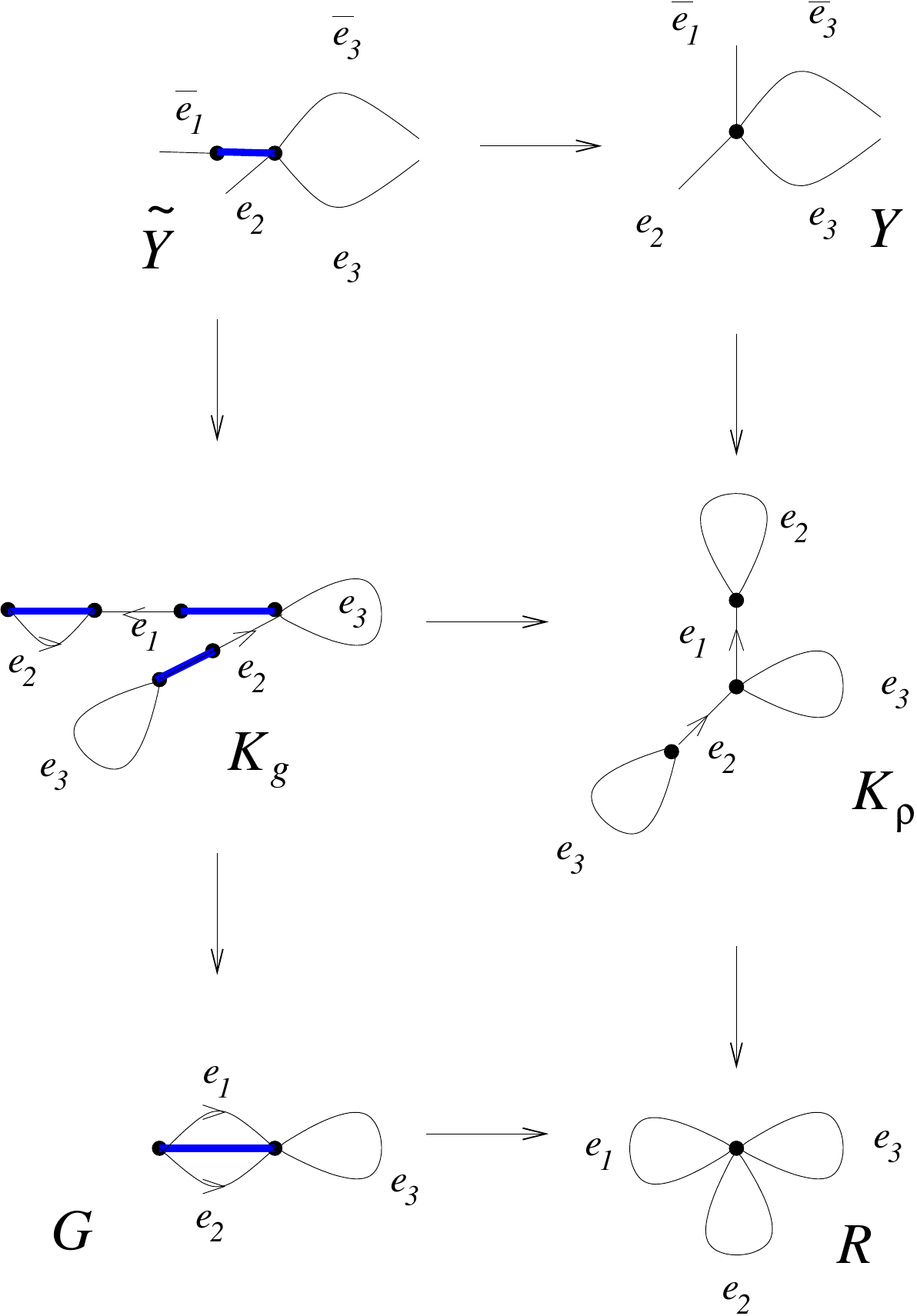}
\caption{
The right side illustrates a rose $\rho=(R,r)$, a Stallings graph $K_\rho$ for an element $K\in\sgs$, and a piece $Y$ with $\partial Y=\{\bar e_1, e_2, e_3, \bar e_3\}$. $(G,g)$ illustrates a point in $St(\rho)$ with a new (blue) edge corresponding to the   ideal edge $E=\{\{\bar e_1,\bar e_2\},\{\bar e_1,\bar e_2\}^c\}$. The partition of $\partial Y$ induced by $E$ is $\{\{\bar e_1\},\{e_2, e_3, \bar e_3\}\}$, $E$ separates $\partial Y$, and $|E|_{\partial Y}=1$. Cf.\ Figure 3 of \cite{kv:contractible}.}
\label{f:ideal}
\end{figure}

\begin{remark}\label{r:construct stallings}
Let $K$ be the element of $\sgs$ represented by $\langle a_1,\dots, a_N\rangle <\f$. The marked Stallings graph $(\Gamma\to \Rn, id_{\Rn})$ representing $K$ can be constructed by taking the core of the graph obtained by starting with the join of loops representing the $a_i$'s and folding; see \cite{js:folding} (or for example \cite{fh:coherence}). A similar statement holds for elements of $X_\sgs$.
\end{remark}

\section{Proof of Theorem~\ref{T1}} \label{s:non-equivariant}

We assume throughout this section that $\group$ is trivial.   After recalling certain key results from  \cite{cv:moduli} and \cite{kv:contractible} (giving references in \cite{kv:contractible}),  we  make the necessary modifications so that they apply in our current context with $\sgs$ replacing $\cC$.  Finally, we verify that our contractible subcomplexes form a \pf\ and apply Proposition~\ref{p:gfp}.

\subsection{Contractibility in Case $\cC$ }
Recall that $L$ denotes the subcomplex of the spine of Outer space spanned by graphs with no separating edges.
By \cite[Proposition~6.1]{kv:contractible}, the spine of Outer space deformation retracts to $L$.  

Fix a sequence $(w_1, \cdots)$ in $\cC$ that includes all elements of $\cC$ (repeats allowed). In \cite{kv:contractible}, Vogtmann takes the sequence to be without repeats and in non-decreasing reduced word length with respect to a fixed basis, but this is not used in her proof. This flexibility, similar to that in \cite{cv:moduli}, will be important for us.  For $\rho\in\cR$, denote by $\|\rho\|$ the element $(\|\rho\|_{w_1}, \cdots)$ of the ordered abelian group $\Z^\cC$. Order the set $\cR$ of marked roses by $\rho<\rho'$ if $\|\rho\|<\|\rho'\|$.

We record three results from \cite{kv:contractible} and then prove two corollaries of those results.

\begin{proposition} [{\cite[Proposition 5.2]{kv:contractible}}] \label{p:well order V}
The set $\cR$ of roses is well-ordered by $<$.
\end{proposition}

 Define  $L_{< \rho} = \cup_{\rho' < \rho }St(\rho')$
  
\begin{lemma}  [{\cite[Corollary 7.2]{kv:contractible}}]  \label{l:factorization}If $\st(\rho) \cap L_{< \rho} $ is not empty then it contains an element of $\cR$ that is obtained from $\rho$ by a Whitehead move.  
\end{lemma}

\begin{proposition} [{\cite[Section 10]{kv:contractible}}]  \label{kv:prop}$\st(\rho) \cap L_{< \rho}$ is either empty or contractible.   
\end{proposition}

The main result of \cite{kv:contractible}, namely contractibility of $L$ and hence the spine of Outer space,  follows easily from Proposition~\ref{kv:prop} and the fact that the spine is connected.  We apply that argument here to conclude that the same is true for each $L_{<\rho}$.  

\begin{corollary}\label{c:L is contractible}  $L_{<\rho}$ is contractible for each non-minimal $\rho \in \cR$.
\end{corollary}

\proof  The proof is by (transfinite) induction in two steps.   Let $C(\rho)$ be the statement that  each component of $L_{<\rho}$ is contractible.  This  obviously holds for   the minimal element of $\cR$ so we may assume that $\rho$ is not minimal and that $C(\rho')$ holds for each $\rho' < \rho$.    If there is a largest element $\rho' \in L_{<\rho}$  then $C(\rho)$ follows from Proposition~\ref{kv:prop} and the inductive hypothesis.  Otherwise, there is an increasing  sequence $\rho_1< \rho_2 <\ldots $ such that $L_{<\rho}$ is the increasing union $\cup_{i=1}^\infty L_{\rho_i}$.  In this case, $C(\rho)$ follows from the inductive hypothesis.  

If the corollary fails, then there exists  $\rho' $ such that $L_{\rho'}$ has at least two components, $A$ and $B$.   An induction argument exactly as above shows that  $A$ and $B$ are contained in separate components of $L_{\rho}$ for all $\rho' < \rho$ and hence contained in separate components of $L$, in contradiction to the fact that $L$ is connected \cite[Propositions~4.1 and 6.1]{kv:contractible}.  
\endproof

 \begin{corollary}  [cf.\ \cite{jhcw:equivalent}] \label{l:whitehead auto}
If $\rho\in\cR$ is not minimal then there  exists $\rho' \in \cR$ such that $\rho' < \rho$  and such that $\rho'$ is obtained from $\rho$ by a Whitehead move. 
\end{corollary}

\proof  Let $\hat \rho$ be the  smallest element of $\cR$ that is greater than $\rho$ in the well-ordering.  Thus  $L_{< \hat \rho} = L_{<\rho} \cup St(\rho)$.   Corollary~\ref{c:L is contractible}.implies that $L_{< \hat \rho}$ is connected and hence that  $L_{<\rho} \cap St(\rho)$ is not empty.  Lemma~\ref{l:factorization} therefore completes the proof.
\endproof

\subsection{Contractibility in Case $\sgs$} \label{s:case sgs} We now extend these results so as to apply to $\sgs$. The only difference between the $\cC$ case  and $\sgs$ case is that some pieces in the latter have vertices of valence bigger than two.  Many of the arguments in the $\cC$ case therefore carry over to the $\sgs$ setting without change.

   Fix a sequence $(\sg_1,\sg_2,\cdots)$ in $\sgs$ that includes every element of $\sgs$.    Update $\|\rho\|$ to now denote the element $(\|\rho\|_{\sg_1},\|\rho\|_{\sg_2},\cdots)$ of $\Z^\bN$. We also update $<$ and write $\rho<\rho'$ if $\|\rho\|<\|\rho'\|$.     

We have the following analog of Proposition~\ref{p:well order V}. 
   
\begin{proposition}
\label{p:gersten order}
$\cR$ with the updated $<$ is well-ordered.
\end{proposition}

\begin{proof}
With the observation that $w\in\cC$ determines an element of $\sg\in \sgs$ with $\|\rho\|_w=\|\rho\|_\sg$, the proof of \cite[Proposition~5.2]{kv:contractible} goes through {\it mutatis mutandis}.
\end{proof}

The following proposition is the analog of Prop~\ref{kv:prop}.
\begin{proposition} \label{p:main gersten}
Let $<$ be the updated well-ordering.  Then, for all $\rho\in\cR$, $\st(\rho) \cap L_{< \rho}$ is either empty or contractible.   
\end{proposition}

\proof    The length function  $\|\tau\|_w$ is introduced in Section 5 of \cite{kv:contractible}.  We have already defined its analog $\|\tau\|_K$.   The statement and proofs in Sections 6--8 of \cite{kv:contractible} apply, using this updated norm, without change.   In particular, Lemma~\ref{l:factorization} of this paper holds for the updated well-ordering.  The  star graph $\Gamma_w$ 
is introduced in Section 9 of \cite{kv:contractible}.   We have already defined its analog $\Gamma_K$.
No changes are required in \cite[Section 10]{kv:contractible}, at the end of which Proposition~\ref{kv:prop} has been shown to follow from a key lemma \cite[Lemma~10.1]{kv:contractible}.

Our final modification concerns \cite[Lemmas 11.1 and  11.2]{kv:contractible}.  The second part of the former is not true in our current context and the given proof of the latter depends on the former.  We state and prove the $\sgs$-analog of \cite[Lemma 11.2]{kv:contractible} below as Lemma~\ref{l:11.2}.    With it  in place,  the proof of  \cite[Lemma~10.1]{kv:contractible} (which appears as the proof of \cite[Lemmas 11.3]{kv:contractible})  goes through without further changes and completes the proof of Proposition~\ref{p:main gersten}.
\endproof

\begin{remark}  \label{rmk:cor still hold} Corollaries~\ref{c:L is contractible} and \ref{l:whitehead auto} remain valid with the updated well-ordering because they follow from Propositions~\ref{kv:prop} and \ref{p:well order V} and Lemma~\ref{l:factorization}.
\end{remark}

 \medskip
 
We recall now a definition from \cite[Section 11]{kv:contractible}.  Given   an arbitrary set $A \subset H$, define $|A| =(|A|_{w_1}, |A|_{w_2},|A|_{w_3},\ldots) \in \Z^{\cC}$ where $|A|_w$ is the number of edges in the star graph $\Gamma_w$  with one endpoint in $A$ and the other in $A^c$.  Equivalently, $|A|_w$ is the cardinality of $\{Y \in \Y: \{A,A^c\} \text{ separates  the endpoints of }Y \}$,   where  $\Y = \{Y\}$  is the set of pieces associated to $w$ in Section~\ref{s:Stallings}.

We modify this to fit our context as follows.  Given $A \subset H$, $K \in \sgs$ and the set of pieces $\Y=\{Y\}$ of the Stallings graph $K_\rho$,  define:  
$$|A|_{\partial Y} =  1 \text{ if } \{A, A^c\} \text{  separates } \partial Y \text {and } 0 \text{ otherwise}$$
$$|A|_K  = \sum_{Y \in \Y} |A|_{\partial Y} \qquad \qquad  |A|_\sgs = (|A|_{K_1}, |A|_{K_2},|A|_{K_3},\ldots) \in \Z^{\sgs}$$

\medskip

The following lemma is {\cite[Lemma~11.2]{kv:contractible}} adapted to our current context.
\begin{lemma}[cf.\ {\cite[Lemma~11.2]{kv:contractible}}]\label{l:11.2}
Let $A$ and $B$ be subsets of $H$. Then $$|A\cap B|_\sgs+|A\cup B|_\sgs\le|A|_\sgs+|B|_\sgs$$
\end{lemma}

\begin{proof}
It is enough to prove the following inequality for each $K \in \sgs$ and each $Y \in \Y$ where $\Y$ is the set of pieces in $K_\rho$:
$$|A\cap B|_{\partial Y}+|A\cup B|_{\partial Y}\le|A|_{\partial Y}+|B|_{\partial Y}$$

We have a partition of $H$ into disjoint subsets $\bZ=A\cap B$, $\bW=\bar A\cap \bar B$, $\bX=A\cap\bar B$, and $\bY=B\cap \bar A$. There are cases depending on how $\partial Y$ intersects these subsets. Taking into account that $\partial Y\not=\emptyset$, that $A$ and $A^c$ determine the same partition of $H$, and that the inequality in the statement is symmetric in $A$ and $B$, we are reduced to the following cases.

\begin{itemize}
\item
$\partial Y$ meets exactly one of $\bZ$, $\bW$, $\bX$, and $\bY$: $0+0\le 0+0$
\item
$\partial Y$ meets exactly two of $\bZ$, $\bW$, $\bX$, and $\bY$:
\begin{itemize}
\item
$\bZ$ and $\bW$: $1+1\le 1+1$
\item
$\bZ$ and $\bX$: $1+0\le 0+1$
\item
$\bX$ and $\bY$: $0+0\le 1+1$
\end{itemize}
\item
$\partial Y$ meets exactly 3 of $\bZ$, $\bW$, $\bX$, and $\bY$:
\begin{itemize}
\item
not $\bX$: $1+1\le1+1$
\item
not $\bZ$: $0+1\le 1+1$
\end{itemize}
\item
$\partial Y$ meets $\bZ$, $\bW$, $\bX$, and $\bY$: $1+1\le 1+1$
\end{itemize}
\end{proof}

As mentioned above, this also finishes the proof of Proposition~\ref{p:main gersten}. \qed

\subsection{Applying Proposition~\ref{p:gfp}} \label{s:applying gfp}

Given $\vec K =(K_1,\cdots,K_N)\in X_\sgs$, choose an extension $K_1,\cdots,K_N, K_{N+1},\cdots$ of $\vec K$ to an infinite sequence  that contains each element of $\sgs$ and then define  a well-ordering $<$ on $\cR$ using this infinite sequence as in Section~\ref{s:case sgs}.     If $\rho \in \cR$ is not minimal in $<$  then by Remark~\ref{rmk:cor still hold}  and Corollary~\ref{l:whitehead auto}  there exists $\rho' \in \cR$ such that  $\rho'   < \rho$ and such that  $\rho'$ differs from $\rho$ by a Whitehead move. Thus, starting with  any $\rho\in\cR$, we can find $\rho_{min}=(R_{min}, r_{min})$.  Define    $\vec z_{min} = \|\rho_{min}\|_{\vec K}=(\|\rho_{min}\|_{K_1}, \ldots,\|\rho_{min}\|_{K_N })$.  Let $\cR_{\vec K}$ be the  initial interval of $<$ such that $\rho\in \cR_{\vec K}$ if and only if $\|\rho\|$ starts with $\vec z_{min}$;  let $L_{\vec K} = \cup_{\rho' \in \cR_{\vec K}}St(\rho')$.  Since $\cR_{\vec K}$ is an initial interval of $\cR$ and since the complement of $\cR_{\vec K}$ in $\cR$ has a first element,  Corollary~\ref{c:L is contractible} implies that $L_{\vec K}$ is contractible.  

By Proposition~\ref{p:gfp}, the following lemma completes the proof of Theorem~\ref{T1}.  
\begin{proposition} \label{presentation family 1}
$\{L_{\vec K}\}$ is a \pf\ in $L$ for the action of $\Out(F_n)$ on $X_\sgs$.
\end{proposition}

\proof We check the conditions in Definition~\ref{defn:gfp}. 

We represent elements of $\Out(\f)$ as homotopy equivalences of $\Rn$ or as automorphisms given in terms of their effects on a basis. In the latter case, two such can be composed and are equal iff one composed with the inverse of the other is inner; something that can be checked on a basis. Thus Definition~\ref{defn:gfp}\pref{item:properties of G} holds.

We represent elements of $X_\sgs$ as marked Stallings graphs $(\Gamma\to \Rn,id_{\Rn})$; see Section~\ref{s:Stallings} and especially Proposition~\ref{p:fix}\pref{i:conj char}. Given $\vec K_1, \vec K_1\in X_\sgs$, and $\phi\in\Out(\f)$, we can compare the Stallings graphs for $\phi(\vec K_1)$ and $\vec K_2$ using Remark~\ref{r:construct stallings}. Thus Definition~\ref{defn:gfp}\pref{item:properties of X} holds.

The action of $\Out(F_n)$ on $L$ preserves the topological types of the vertices of $L$ and so does not invert edges.  Thus \pref{item:no inversion} holds.

For \pref{item:Gvw}, suppose that we are given vertices $\tau = (G,g)$ and $\tau'=(G',g')$ of $L$.  Decide if   $G$ and $G'$ are cellularly isomorphic.  If {\tt NO}  then there does not exist $\theta \in \Out(F_n)$ such that $ \tau \cdot  \theta = \tau'$.  If {\tt YES}  then the finite set of $\theta\in\Out(\f)$ satisfying $\tau\cdot \theta=\tau'$ can be algorithmically listed.  Indeed, such a $\theta$ is determined by a homotopy equivalence of the form $(g')^{-1}hg$ where $h:G\to G'$ is a cellular isomorphism.  Thus \pref{item:Gvw} holds.

  For \pref{item:natural}, let $\theta\in\Out(\f)$ and $\rho\in\cR$. By Remark~\ref{r:fix}, $\|\rho\|_{\theta\vec K}=\|\rho\theta\|_{\vec K}:=\|\theta^{-1}\rho\|_{\vec K}$. 
  Hence $\rho\in\theta\cR_{\vec K}$ iff $\rho\in \cR_{\theta\vec K}$. Thus \pref{item:natural} holds.

It remains to check \pref{item:domain for Lx}. For $S'\subset\cR$, let $[S']$ denote the set of equivalence classes of the Stallings graphs $\{\vec K_{\rho}\mid \rho\in S'\}$; see Section~\ref{s:Stallings}.  We inductively construct  a finite set $S$ of $\Fix(\vec K)$-orbit representatives for $\cR_{\vec K}$. Start with $S_1:=\{\rho_{min}\}$. Given $S_i$, let $W(S_i)$ be the subset of $\cR_{\vec K}$ consisting of elements obtained from elements of $S_i$ by a Whitehead move. If $[W(S_i)]=[S_i]$ then $S:=S_i$. Otherwise, $S_{i+1}:=S_i\cup W(S_i)$. (We could be more efficient here by only adding new equivalence classes.) This process ends since there are only finitely many equivalence classes with a prescribed volume.

We claim that the union $U$ of the $\Fix(\vec K)$-orbits of $S$ equals $\cR_{\vec K}$. Indeed, since $\cR_{\vec K}$ is $\Fix(\vec K)$-invariant and $S\subset \cR_{\vec K}$, we have that $U\subset \cR_{\vec K}$. For the other inclusion, recall that $\cR_{\vec K}$ is an initial interval of the total order $<$ on $\cR$. Suppose that $\rho\in \cR_{\vec K}\setminus U$ is minimal. By Remark~\ref{rmk:cor still hold} there is $\rho'\in \cR_{\vec K}$ obtained from $\rho$ by a Whitehead move such that $\rho'<\rho$. So there is $\theta\in\Fix(\vec K)$ such that $\rho'\theta\in S$. Hence (see Section~\ref{s:background}) $\rho\theta$ differs from an element of $S$ by a Whitehead move. By the definition of $S$, there is $\rho''\in S$ such that $\vec K_{\rho''}=\vec K_{\rho\theta}$. By Proposition~\ref{p:fix}\pref{i:fix orbits} $\rho\theta$, hence also $\rho$, is in the $\Fix(\vec K)$-orbit of $\rho''\in S$, a contradiction.
It follows that $\st(U) = L_{\vec K}$. Defining  $D_{\vec K}=\cup_{\rho\in S} St(\rho)$, we see that \pref{item:domain for Lx} holds.
\endproof

\section{From \cite{kv:equivariant}}\label{s:equiWhitehead}
The rest of the paper is devoted to proving our main result, Theorem~\ref{T2}.  Now the key work is \Krstic\ and Vogtmann's  \cite{kv:equivariant} and we mimic proofs there.
As in the proof of Theorem~\ref{T1}, the main work is showing that a certain subcomplex of the spine of Outer space is contractible.  As in the proof of Theorem~\ref{T1}, we  follow the proof of contractibility when $\sgs$ is replaced by $\cC$, making modifications as necessary for our expanded context.

We start by reviewing needed background material; see especially \cite[Sections  3(B) and 3(C)]{kv:equivariant}. 
Fix a finite subgroup $\group < \Out(\f)$ and define the {\it equivariant spine} to be the subcomplex of the spine of Outer space that is fixed by the action of $\group$. 
 By definition then, the vertices of the equivariant spine are equivalence classes of marked {$\group$-graphs $\Gmg =(\gph, s)$,  i.e.\ $\gph$ is equipped with an action $h:\group\to\Aut(\gph)$, such that if $f_x:R_n\to R_n$ represents $x\in  \group$ then $h(x)\circ s$ is homotopy equivalent to $s\circ f_x$.  Here $\Aut(\gph)$ is the (finite) isometry group of $\gph$ where we have identified each edge of $\gph$ with the unit interval. It is required that $\group$ inverts no edge of $\gph$, a requirement that can be achieved by adding midpoints of inverted edges if necessary. Collapsing $\group$-equivariant forests gives the set of vertices of the equivariant spine a partial order that agrees with the one inherited from the spine of Outer space.  We view the equivariant spine as the geometric realization of this poset.

An edge of a $\group$-graph $\gph$ is {\it inessential} if it is contained in every maximal $\group$-invariant forest of $\gph$. $\gph$ is {\it essential} if no edge is inessential, all vertices have valence at least two and the two edges terminating at a bivalent vertex are in the same $\group$-orbit.
 $L_\group$ denotes the subcomplex of $K_\group$ spanned by essential graphs.  $L_\group$ is a deformation retract of the equivariant spine  \cite[Proposition~3.3]{kv:equivariant}.

The role of roses in this setting is played by {\it reduced} $\group$-graphs, i.e.\ vertices $(\gph,s)\in L_\group$ without a non-trivial $\group$-forest. The set of reduced $\group$-graphs is denoted $\cR_\group$.  Every element of $L_\group$ is contained in the  star of an element of $\cR_\group$. 

If $f: \gph\to \gph'$ is a homotopy equivalence of reduced $\group$-graphs of the form $\gph\leftarrow G\to \gph'$ where $G$ is a $\group$-graph and both maps collapse an edge orbit then we say that $(\gph', f\circ s)$ is obtained from $(\gph, s)$ by a {\it Whitehead move}.

\begin{proposition} [{\cite[Proposition 5.11]{kv:equivariant}}] Any two elements of $\cR_\group$ are connected by a sequence of Whitehead moves.
\end{proposition}

\begin{corollary}  [{\cite[Corollary 5.12]{kv:equivariant}}] $ L_\group$ is connected.
\end{corollary}

 \medskip
 
  Following \cite[Section 6]{kv:equivariant}, we  define a norm on $\cR_\group$ and a $\group$-star graph.    Given $K \in \sgs$ and   $\Gmg =(\gph, s)$  define 
  $$\|\Gmg\|_{\group, K} = \sum_{x \in \group} \|\Gmg\|_{xK} =  |\group|\cdot \|\Gmg\|_K$$  
  where  the volume  $ \|\Gmg\|_K$ is defined in Section~\ref{s:Stallings}.
  As in the proof of Theorem~\ref{T1}, we fix a sequence $\vec K = (K_1,K_2,\ldots )$ in $\sgs$ that includes all elements of $\sgs$.  For each   $\Gmg\in \cR_\group$, define 
  $$\|\Gmg\|_\group = (\|\Gmg\|_{\group,K_1}, \|\Gmg\|_{\group,K_2},\ldots) \in \Z^{\sgs}$$
   and order $\cR_\group$ by 
   $\Gmg <_\group \Gmg'$ if $\|\Gmg\|_\group < \|\Gmg'\|_\group$.
  
  \medskip 
   
The Stallings graph $K_\Gmg$, and its decomposition into pieces $\sfY = \{Y_i\}$, is defined as in Section~\ref{s:Stallings} with the role of $H$ being played by the set $E(\gph)$ of oriented edges of $G$.  In particular,  the elements of $\partial Y_i$ are labeled by distinct elements of  $E(\gph)$.
  The $\group$-star graph for  $K \in \sgs$ and   $\Gmg =(\gph, s)$ is defined to be  the graph formed by \lq superimposing\rq\ the star graphs of $xK$   for each $x\in \group$ and so has the following description.    Begin, as in the non-equivariant case, with a set $V$  of vertices  labeled by the elements of  $E(G)$. Then, for each $x \in \group$ and $Y\in \sfY$, add a copy of $xY$ with $  \partial (xY)$ attached to $V$ according to its labels.

 Propositions~\ref{p:well order equivariant} and \ref{kv:prop equivariant} below are the analogs of Propositions~\ref{p:gersten order} and Lemma~\ref{l:factorization}.    Their $\cC$ versions appear as Proposition~6.3 and Lemma~6.5 of \cite{kv:equivariant}.  The proofs  given there carry over to the  $\sgs$ case without modification.

\begin{proposition}   \label{p:well order equivariant}The set $\cR_\group$ of roses is well-ordered by $<_\group$.
\end{proposition}

Define  $L_{<_\group}( \Gmg) = \cup_{ \Gmg'  < \Gmg  }St(\Gmg)$

\begin{lemma}  \label{l:Gfactorization}If $\st(\Gmg) \cap L_{<_\group}( \Gmg) $ is not empty then it contains an element of $\cR$ that is obtained from $\Gmg$ by a Whitehead {move}.  
\end{lemma}

We denote by $\sigma_{min}$ the minimal element of $\cR_\group$ with respect to  $<_\group$ and by $\cR_{\group,\vec K}$ the subset of $\cR_{\group}$ of $\sigma$ such that $\|\sigma\|_{\group,K_i}=\|\sigma_{min}\|_{\group, K_i}$ for $1\le i\le N$. By $L_{\group,\vec K}$ denote the union of the stars (in $L_\group$) of the elements of $\cR_{\group,\vec K}$. 
This notation assumes that we have already specified an element $\alpha\in\Hom(\group,\Out(\f)\big)$ and $\vec K\in X_\sgs$. If we want to stress the dependence on $\alpha$ and $\vec K$ then we will write $\cR_{(\alpha,\vec K)}$ for $\cR_{\group,\vec K}$ and $L_{(\alpha,\vec K)}$ for $L_{\group,\vec K}$.

As we did with Theorem~\ref{T1} (see Section~\ref{s:applying gfp}), to prove Theorem~\ref{T2} we will apply Proposition~\ref{p:gfp} -- this time:
\begin{itemize}
\item
$X:=\Hom\big(\group,\Out(\f)\big)\times X_\sgs$;
\item
$G:=\Out(\f)$;
\item
$L$ is again the subcomplex (still denoted $L$) of the spine of Outer space consisting of marked graphs with no separating edges; and
\item
for $x=(\alpha,\vec K)\in X$, $L_x:=L_{(\alpha,\vec K)}$. 
\end{itemize}

\noindent With only the obvious modifications, which we leave to the reader, the arguments of Section~\ref{s:applying gfp} reduce the proof of Theorem~\ref{T2} to the  following analog of Proposition~\ref{p:main gersten}. (To verify Definition~\ref{defn:gfp}\pref{item:properties of X}, we view an element of $\Hom\big(\group,\Out(\f)\big)$ as an action $\group\to \Aut(\Gamma)$ for some marked graph $\Gamma$, which is possible by \cite{mc:finite, bz:uber}.)

\begin{proposition}    \label{kv:prop equivariant}$\st(\rho) \cap L_{\group,< \rho}$ is either empty or contractible.   
\end{proposition}

The $\cC$ version of Proposition~\ref{kv:prop equivariant} is proved in \cite[Sections~7 and 8]{kv:equivariant}.  
As in the non-equivariant case, most of the arguments extend to our current context without change.  The two exceptions are Lemmas~7.1 and 7.2 of \cite{kv:equivariant}, which appear as Lemmas~\ref{l:7.2} and \ref{l:7.1} in this paper.  We state and prove these lemmas in Section~\ref{s:last}.   Once this is done, the arguments in \cite[Section~8]{kv:equivariant} carry over to complete the proof of  Proposition~\ref{kv:prop equivariant} and hence the proof of Theorem~\ref{T2}.  

In summary, we have proved Theorem~\ref{T2} modulo verifying Lemmas~\ref{l:7.2} and \ref{l:7.1}.

\section{The Combinatorial Lemmas } \label{s:combinatorial}
In this section we complete the proof of  Theorem~\ref{T2} by proving  Lemmas~\ref{l:7.2} and \ref{l:7.1}, the analogs of \cite[Lemmas~7.2 and 7.1]{kv:equivariant}.

The following notation is used throughout this section.

Fix $\rho \in  \cR_\group$ and let $E(\Gamma)$ be the set of oriented edges in $\Gamma$. The definition of an {\it ideal edge} $\alpha$ of $\Gamma$ in the equivariant setting is given in \cite[Section~5.A]{kv:equivariant}. In particular $\alpha\subset E_v\subset E(\Gamma)$ where $v$ is a vertex of $\Gamma$ and  $E_v$ is the set of oriented edges of $\Gamma$ terminating at  $v$. We repeatedly   use  the following property of an ideal edge. 
\begin{equation}
\label{e:edge}
\tag{edge}
\alpha\cap x\alpha\not=\emptyset, x\in \group \implies x\in\Stab(\alpha)
\end{equation}

\begin{remark}   Following \Krstic\  and Vogtmann, let $P$ be    the subgroup of $\Stab(v)$ generated by the stabilizers of the oriented edges in $\alpha \subset E_v$.   The second condition in their definition of ideal edge is that $\group \cap \alpha = P e$ for all $e \in \alpha$.  From this it follows easily that $P = \Stab(\alpha)$ and that (edge)  is satisfied.
\end{remark}

\def\YY{Z}
 For subsets $A$ and $\YY$ of $E(\Gamma)$, define $|A|^*_\YY $  to be  $1$ if $\{A,A^c\}$ separates\footnote{equivalently $\YY$ meets both $A$ and $A^c$} $\YY$ and $0$ otherwise. The superscript $*$ indicates that we are ignoring the group action.   For any subgroup  $\group' < \group$, define $$|A|_{\YY,\group'} =  \sum_{g \in \group'}|A|_{g\YY}^*$$
 and $$  |A|_\YY =  |A|_{\YY,\group}$$

\begin{remark}  \label{rem:implicit} $|gA|^*_{g\YY} = |A|^*_\YY$ for all $g \in \group$ and all  $A,\YY \subset E(\Gamma)$.  Thus, if $gA= A$ then $|A|^*_{g\YY} = |A|^*_\YY$.   Similarly, $|A|_{\YY,\group'} = |A|_{g\YY,\group'}$ for all $g \in \group'$.  We use these facts repeatedly, usually without comment.   
\end{remark}

\subsection{The First Combinatorial Lemma}  \label{s:first combinatorial} 
The goal of this section is to prove the combinatorial lemma that underlies our  analogue (Lemma~\ref{l:7.1}) of \Krstic-Vogtmann's Lemma~7.1 \cite{kv:equivariant}. 

\begin{notn} \label{notn:intermediate}
We assume throughout this subsection that    $\alpha ,\beta\subset E(\Gamma)$ are ideal edges  whose stabilizers $P = \Stab(\alpha)$ and $Q = \Stab(\beta)$ have indices $p$ and $q$ in $\group$, respectively.  We also assume 
such that: 
 \begin{enumerate}
 \item $\alpha \cap \beta \ne \emptyset$
 \item $P \le Q$
 \item $\alpha\cap g\beta\not=\emptyset, g\in \group\implies g\in Q $ \label{item:simple}
 \end{enumerate}
Let  $\gamma :=  \alpha\cap\beta;  \ \gamma'  := Q \gamma  = Q\alpha \cap \beta; \ A := \alpha - \gamma$; and  $B :=
  \beta \cup Q \alpha = \beta + Q A$. (Disjointness of $\beta$ and $Q A$ as well as the  final equality follows from the fact that $ \beta$ and $\beta^c$ are $Q$-invariant.)  In particular,
 \begin{itemize}
\item
$Q\alpha$, $\beta$,  $B$, $B^c$  are all $Q$-invariant.
 \end{itemize}

\begin{figure}[h!] 
\centering
\includegraphics[width=.25\textwidth]{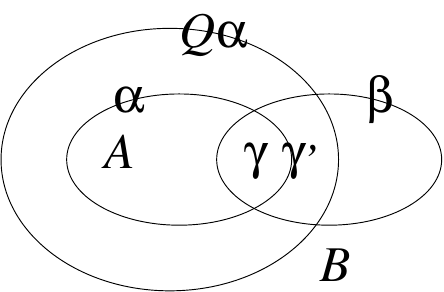}
\caption{}
\label{f:9}
\end{figure}
\end{notn}

The following lemma is the main result of this subsection.  
\begin{lemma} \label{lem:intermediate}  For any $\YY \subset E(\Gamma)$, 
\begin{equation}
\label{eq:intermediate}
p|\gamma|_{\YY,Q}+q|B|_{\YY,Q}\le p|\alpha|_{\YY,Q}+q|\beta|_{\YY,Q}
\end{equation}
 \end{lemma}
 
 \vspace{.2in}
 
The proof of Lemma~\ref{lem:intermediate} occupies the rest of this subsection.
  We will need the following property of $B$ and $Q$.

\begin{lemma}  
$g\in \group$, $B\cap gB\not=\emptyset\iff g\in Q$. 
\end{lemma}
\begin{proof}
$\Leftarrow$: $B$ is  $Q$-invariant.

$\Rightarrow$: Assuming that $y,g(y) \in B$, we must show that $g \in Q$.  Case 1: Assume $y\in \beta$. If $gy\in \beta$ then $g\in Q$ by  the \pref{e:edge} property of $\beta$. If $gy\in Q\alpha$ then there is $x\in Q$ such that $x^{-1}gy\in\alpha$ and \ref{notn:intermediate}\pref{item:simple} implies that  $x^{-1}g\in Q$ and so $g\in Q$. Case 2: Assume $y\in Q\alpha$ and hence  there is $x \in Q$ such that $x^{-1} y \in \alpha$. If $gy\in\beta$ then $x^{-1}g^{-1}(gy)\in\alpha$ and \ref{notn:intermediate}\pref{item:simple} implies that $x^{-1}g^{-1}\in Q$ and so $g\in Q$. If $gy\in Q\alpha$ then there is  $x'\in Q$ such that  $x'^{-1}gx(x^{-1}y)= x'^{-1}gy\in\alpha$.
By the \pref{e:edge} property of $\alpha$,   $x'^{-1}gx\in P$ and so $g\in Q$.
\end{proof}

\begin{notn}\label{n:7.1}
The disjoint regions in Figure~\ref{f:9} are $QA-A $, $A$, $\gamma$, $\gamma'-\gamma$, $\beta-\gamma'$, $B^c$.  Note that all of these regions are $P$-invariant. We use a sequence of 0s and 1s to indicate which regions contain an element of $\YY$. For example, $011000$ denotes that $\YY$ meets $A$ and $\gamma$ and does not meet any other component.
\end{notn}

We now prove Lemma~\ref{lem:intermediate} by a case analysis.

 \begin{itemize}
\item
$\YY\subset B^c$: Then $Q\YY\subset B^c$ and \ref{lem:intermediate}\pref{eq:intermediate} becomes $0\le 0$.
\item
 $\YY$ meets both $B$ and $B^c$:    From $\YY \cap B^c \ne \emptyset$ it follows that $|\gamma|^*_{g\YY}\le|\alpha|^*_{g\YY}$  for all $g \in Q$; in particular, $|\gamma|_{\YY,Q}\le|\alpha|_{\YY,Q}$.  
 If $\YY$ meets $\beta$ then $|B|^*_{g\YY}=|\beta|^*_{g\YY} = 1$ for all $g \in Q$.  Thus $|B|_{\YY,Q}=|\beta|_{\YY,Q} = |Q|$ and \ref{lem:intermediate}\pref{eq:intermediate} is satisfied. 
  We may therefore assume that  $\YY$ misses $\beta$ and hence meets $Q\alpha$.  After replacing $\YY$ by $g\YY$ for some $g \in Q$, we may  assume that $\YY$ meets $\alpha$.  Since $\YY$ misses $\beta$, $g\YY$ misses $\gamma$ for all $g \in Q$ and $|\gamma|_{\YY,Q} = 0$.  Since $g\YY$ meets $\alpha$ for all $g \in P$,  $p|\alpha|_{\YY,P} =p|P| = |\group|$.   \ref{lem:intermediate}\pref{eq:intermediate}  therefore follows from   $q|B|_{\YY,Q} \le q |Q| = |\group|$. 
 
\item
$\YY\subset B$: In this case $|B|_{\YY,Q}=0$. If $|\gamma|_{\YY,Q}=0$ then the left hand side of  \ref{lem:intermediate}\pref{eq:intermediate} is 0 and we are done. Hence we may further assume that there exists $x \in Q$ with $ |\gamma|^*_{x\YY} = 1$.  After replacing $\YY$ by  $x\YY$,   we have  $|\gamma|_\YY = 1$.  Thus, $\YY$ intersects both $\gamma$ and $\gamma^c$.   If $\YY$ intersects $\beta-\gamma'$, which is a $Q$-set  in the complement of $Q\alpha$,  then $g\YY$ intersects $\alpha ^c$ for each $g \in Q$.  In this case, $|\gamma|^*_{g\YY}\le |\alpha|^*_{g\YY}$ for all $g\in Q$ and we are done. So we additionally assume that $\YY$ doesn't meet $\beta-\gamma'$.  Since $\YY \subset B$, we have $\YY\subset Q\alpha$.  To summarize, we are left with proving \ref{lem:intermediate}\pref{eq:intermediate} in the case that   $\YY \subset Q\alpha$ intersects both $\gamma$ and  $\gamma^c$.  Using Notation~\ref{n:7.1}, $\YY$ has the form $\delta_1  \delta_2 1  \delta_400$ with $\max(\delta_1,\delta_2,\delta_4)=1$ where $\delta_i$ denotes a 1 or a 0.
 \end{itemize}
 
 The following chart shows the required case analysis.  Details are given in Lemma~\ref{elaborate2} and the paragraph that follows it.  
  
\vskip 6pt
\centerline{
\begin{tabular}{ |c||c||c|c|c|c|}
 \hline
& form & $p|\gamma|^*$ & $q|\beta\cup Q\alpha|^*$ & $p|\alpha|^*$ &$q|\beta|^*$ \\
 \hline
 \hline
$P\YY$ & 
$\delta_1\delta_21\delta_400$  & $|P|p$  & 0 & $\max(\delta_1,\delta_4)|P|p$ & $\max(\delta_1,\delta_2)|P|q$ \\
 \hline
$(Q-P)\YY$ & $\delta_1'\delta_2'?100$  & $\Delta$ & 0 & $\ge\Delta$ & $\max(\delta_1',\delta_2')|Q-P|q$\\
 \hline
\hline
column totals: $Q\YY$ & & $|\group|+\Delta$   & $0$ & $ \max(\delta_1,\delta_4)|\group|+(\ge\Delta)$ &$\max(\delta_1,\delta_2)|\group|$ \\
 \hline
\end{tabular}
}
\vskip 12pt
\noindent Here $\max(\delta_1,\delta_2)=\max(\delta'_1,\delta'_2)$ and ? denotes an indeterminate 0 or 1 that doesn't factor into the computation. Since $\YY$ intersects $\gamma^c$, $\max(\delta_1,\delta_2,\delta_4)=1$, and so   \ref{lem:intermediate}\pref{eq:intermediate} is satisfied.

\begin{lemma} \label{elaborate2} Suppose that  $\YY \subset Q\alpha$ intersects both $\gamma$ and  $\gamma^c$.  Define $\delta_1$ [resp. $\delta_2, \delta_4$] to be $1$  if $\YY $ intersects $ QA- A$ [resp. $A$, $\gamma'-\gamma$] and $0$ otherwise.  
 \begin{enumerate}[(a)]
 \item  For all $g \in Q$, \  $|B|^*_{g\YY} = 0$.
   \item For all $g \in Q$, \ $|\beta|^*_{g\YY} = 1 \Longleftrightarrow  \YY $ intersects $QA \Longleftrightarrow  \max(\delta_1,\delta_2) = 1$. 
 \item  For all $g \in P$:\ \      $|\gamma|^*_{g\YY} = 1$;  and \     
  $|\alpha|^*_{g\YY} = 1\Longleftrightarrow \YY $ intersects $Q\alpha - \alpha\Longleftrightarrow \max(\delta_1,\delta_4) =1    $
  \item  If $g \in Q - P$ then $|\gamma|^*_{g\YY}\le |\alpha|^*_{g\YY} $
\end{enumerate}
\end{lemma}

\proof  Since $B$ and $\beta$ are $Q$-invariant,  $|B|^*_{g\YY}$   and  $|\beta|^*_{g\YY} $ are independent of $g \in Q$ and it suffices to verify (a) and (b) with $g\YY$ replaced by $\YY$, in which case (a) and (b) follow from the hypotheses of the lemma and the definitions.  Item (c) is proved similarly, noting that $\gamma$ and $\alpha$ are $P$-invariant. For $g \in Q- P$,  the (edge) property of $P$ implies that $g\YY \cap \alpha^c \ne \emptyset$ and hence that  $|\gamma|^*_{g\YY}\le |\alpha|^*_{g\YY} $.
\endproof

Applying Lemma~\ref{elaborate2} and $\max(\delta_1,\delta_2,\delta_4) \ge 1$,  

$$p|\alpha|_{\YY,Q}- p|\gamma|_{\YY,Q}+q|\beta|_{\YY,Q}-q|B|_{\YY,Q} $$ 
$$ \ge p|\alpha|_{\YY,P} -p|\gamma|_{\YY,P}+q|\beta|_{\YY,Q} \qquad \text{ using (a) and  (d) }$$
$$ =   p|\alpha|_{\YY,P}+q|\beta|_{\YY,Q} -p|P| \qquad \text{ using (c)} $$
$$ =   p|P| \max(\delta_1,\delta_4)+q|Q| \max(\delta_1,\delta_2) -|\group| \qquad \text{ using (b) and  (c) }$$
$$ =   |\group| \max(\delta_1,\delta_4)+|\group| \max(\delta_1,\delta_2) -|\group| \ge (\max(\delta_1,\delta_2,\delta_4) -1) |\group| \ge0$$
This completes the proof in the case that $\YY \subset B$ and so also   the proof of Lemma~\ref{lem:intermediate}. \qed

\subsection{The Second Combinatorial Lemma}   \label{s:second combinatorial}

The goal of this section is to prove the combinatorial lemma that underlies \cite[Lemma~7.2]{kv:equivariant}.

\begin{lemma}[cf.\ {\cite[Lemma~7.2]{kv:equivariant}}]\label{l:7.2}  Suppose that  $x \in \group$ and that  $\alpha ,\beta\subset E(\Gamma)$ are ideal edges  whose stabilizers $P = \Stab(\alpha)$ and $Q = \Stab(\beta)$ have indices $p$ and $q$ in $\group$, respectively.    Let 
$$\gamma = \alpha \cap Px\beta = P(\alpha \cap x\beta) \qquad \text{ and } \qquad \gamma' = \beta \cap Qx^{-1}\alpha = Q(\beta \cap x^{-1} \alpha)$$  Then   for all $\YY \subset E(\Gamma)$,
 $$ p|\alpha-\gamma|_\YY+q|\beta-\gamma'|_\YY\leq p|\alpha|_\YY+q|\beta|_\YY$$
 \end{lemma}

\begin{figure}[h!] 
\centering
\includegraphics[width=.27\textwidth]{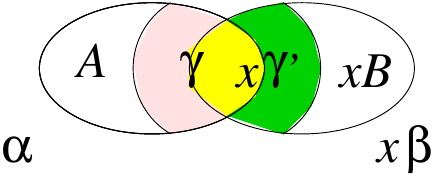}
\caption{The left ellipse represents $\alpha$ and the right $x\beta$. The union of the pink and yellow regions represents $\gamma$ and the union of the yellow and green $x\gamma'$. $A=\alpha-\gamma$ and $xB=x\beta-x\gamma'$. }
\label{f:10}
\end{figure}

\proof  Letting $A:= \alpha-\gamma$ and $B:= \beta-\gamma'$, our desired inequality is   
\begin{equation} \label{e:7.2}
 p|A|_{\YY}+q|xB|_{\YY}\leq p|\alpha|_{\YY}+q|\beta|_{\YY}
\end{equation}
The proof of \pref{e:7.2} occupies the rest of this subsection.
  We start with some observations on the action of $\group$ on the sets in Figure~\ref{f:10}.

\begin{itemize}
\item  $x\gamma' = xQ(\beta \cap x^{-1}\alpha) =  Q^x(x\beta \cap \alpha)$ where $Q^x = xQx^{-1} = \Stab(x \beta)$.
\item By construction,
the six sets $A, \gamma - (\alpha\cap x\beta), \alpha \cap x \beta,  x\gamma' - (\alpha\cap x\beta), xB =x\beta -x\gamma'$ and $(\alpha \cup x \beta)^c$ shown in  Figure~\ref{f:10} are disjoint. 
\item
By definition, $\alpha$, and so also $\alpha^c$, is $P$-invariant; $\gamma$, hence also $A$, is $P$-invariant. Similarly, $x\beta$, $(x\beta)^c$, $x\gamma'$, and $xB$ are $Q^x$-invariant.
\item
$\group (\gamma\cup x\gamma')$ is disjoint from  $A$ and $xB$ (which are disjoint from each other).
 Indeed,
\begin{itemize}
\item
$\group\gamma$, $A$, and $xB$ are pairwise disjoint:
\begin{itemize}
\item
$g\gamma\cap A\not=\emptyset\implies g\in P$ (by the \pref{e:edge} property of $\alpha) \implies g\gamma=\gamma\implies\gamma\cap A\not=\emptyset$, contradiction.
\item
$g\gamma\cap xB\not=\emptyset \implies gP(\alpha \cap x \beta) \cap x B \ne \emptyset \implies\exists g'\in gP \subset \group, b\in\alpha\cap x\beta$ such that $g'b\in xB$. Thus $g'\in Q^x$ (by the \pref{e:edge} property of $x\beta$) and so $g'b\in x\gamma'$. This gives that $x(\gamma'\cap B)\not=\emptyset$, which is a contradiction.
\end{itemize}
\item
Symmetrically $\group\gamma'=\group x\gamma'$, $A$, and $xB$ are pairwise disjoint. 
\end{itemize}
\end{itemize}

We now use these observations to verify \pref{e:7.2}. 

If 
\begin{equation}\label{e:Y} 
 p|A|^*_{g\YY}+q|xB|^*_{g\YY} \le p|\alpha|^*_{g\YY}+q|x\beta|^*_{g\YY}
 \end{equation} 
 holds for all $g \in \group$ then  we are done, so assume that \pref{e:Y} fails 
    for some $g \in \group$.  Thus, either $|A|^*_{g\YY} = |\alpha-\gamma|^*_{g\YY}> |\alpha|^*_{g\YY}$ or  $|xB|^*_{g
    \YY} = |x \beta-x \gamma'|^*_{g\YY}> |x\beta|^*_{g\YY}$.   By replacing $\YY$ with $g\YY$ and reversing the roles of $\alpha$ and $\beta$ if necessary, we may assume that 
$|A|^*_{\YY} = |\alpha-\gamma|^*_{\YY}> |\alpha|^*_{\YY}$.   It follows that  $\YY\subset\alpha$ and that  $\YY$ meets  both $A$ and $\gamma = P(\alpha \cap x\beta)$.   By replacing $\YY$ with $g\YY$ for some $g\in P$,  we may further assume $\YY\subset\alpha$ and $\YY$ meets both $A$ and $\alpha\cap x\beta$ and so also intersects $x\beta$. 

\begin{notn}\label{n:7.2}
The disjoint regions in Figure~\ref{f:10} are $A$, $\gamma-x\beta$, $\alpha\cap x\beta$, $x\gamma'-\alpha$, $xB$, and $(\alpha\cup x\beta)^c$. We use a sequence of 0s and 1s to indicate which regions contain an element of $\YY$. For example, $101000$ denotes that $\YY$ meets $A$ and $\alpha\cap x \beta$ and does not meet any other component. 
\end{notn}

With this notation, we are reduced to proving  \pref{e:7.2} assuming  that $\YY$ has the form 101000 or 111000. The cases are summarized in  the following chart, where $\Delta$ and $\Delta'$ represent nonnegative quantities.  We elaborate in Lemma~\ref{elaborate}.
\vskip 12pt
\centerline{
\begin{tabular}{ |c||c||c|c|c|c|}
 \hline
& form & $p|A|^*$ & $q|B|^*$ & $p|\alpha|^*$ &$q|\beta|^*$ \\
 \hline
 \hline
$(P\cap Q^x)\YY$ & 101000 or 111000 & $|P\cap Q^x|p$  & 0 & 0 & $|P\cap Q^x|q$ \\
 \hline
$(P- Q^x)\YY$ &110000 or 111000 & $|P-Q^x|p$ & 0& 0 & $\Delta$\\
 \hline
 $(Q^x- P)\YY$  & 000101 & 0 & 0 & 0 & $|Q^x-P|q$ \\
\hline
$\big(\group-(P\cup Q^x)\big)\YY$ & in $\alpha^c$ and meets $(\alpha\cup x\beta)^c$ & 0  & $\le \Delta'$ & 0 & $\Delta'$\\
\hline
\hline
column totals: $\group \YY$& & $|P|p=|\group|$   & $\Delta'$ & 0 &$|\group|+\Delta+\Delta'$ \\
 \hline
\end{tabular}
}

\vskip 12pt

\begin{lemma} \label{elaborate} Suppose that  $\YY \subset \alpha$ intersects both $A$ and  $x\beta$.  
\begin{enumerate}[(a)]
\item  If $g \in (P\cap Q^x$) then $|A|^*_{g\YY} = |x \beta|^*_{g\YY} = 1$  and   $|\alpha|^*_{g\YY} = |xB|^*_{g\YY} = 0$.
\item If $g \in (P  \cap (Q^x)^c$) then $|A|^*_{g\YY} = 1$ and $ |x B|^*_{g\YY}  = |\alpha|^*_{g\YY} = 0$.
\item If $g \in (Q^x \cap P^c$) then $|A|^*_{g\YY} =  |x B|^*_{g\YY}  = |\alpha|^*_{g\YY} = 0$ and $ |x \beta|^*_{g\YY} = 1$.
\item If $g \in \group -(P\cup Q^x )$ then $|A|^*_{g\YY}   = |\alpha|^*_{g\YY} = 0$ and $ |x B|^*_{g\YY}  \le |x \beta|_{g\YY}$.
\end{enumerate}
\end{lemma}

\proof   We begin with a pair of preliminary observations:   $|A|^*_{\YY}  =  |x \beta|^*_{\YY}  =1$ because $\YY$ intersects both $A$ and $x\beta$, which are disjoint;        $|\alpha|^*_{\YY} = |xB|^*_{\YY} = 0$ because  $\YY \subset \alpha$ and $\YY \cap xB = \emptyset$.

If $g \in P\cap Q^x$ then  $A, xB, \alpha$ and $x\beta$ are all $g$-invariant.  Item (a) therefore follows from our preliminary observations and Remark~\ref{rem:implicit}.

If $g \in P \cap (Q^x)^c$ then  $|A|^*_{g\YY} = |A|^*_{\YY}  =1$  and  $|\alpha|^*_{g\YY} =|\alpha|^*_{\YY} = 0$ because $\alpha$ and $A$ are $g$-invariant.  From $xB \cap \alpha = \emptyset$ and $g^{-1}(\alpha) = \alpha$,  it follows that $g^{-1}(xB) \cap \alpha = \emptyset$ and so $|xB|^*_{g\YY} =  |g^{-1}xB|^*_{\YY} = 0$. This completes the proof of (b).

If $g \in Q^x \cap P^c$ then $|xB|^*_{g\YY} =|xB|^*_{\YY} = 0$  and  $|x\beta|^*_{g\YY} = |x\beta|^*_{\YY} =1$ because $xB$ and $x\beta$ are $g$-invariant.  Property (edge) implies that $g \alpha$ is disjoint from $\alpha$ and hence that $g\YY$ is disjoint from $\alpha$ and $A$.  It follows that   $|A|^*_{g\YY} =   |\alpha|^*_{g\YY} = 0$ completing the proof of (c).

If $g  \in \group - (P\cup Q^x) $ then $|A|^*_{g\YY} =   |\alpha|^*_{g\YY} = 0$ as in  case (c).  Moreover, $g\YY $ intersects $g(\alpha \cap x \beta)$  and so intersects the complement of $x \beta$.  It follows    that $ |x B|_{g\YY}  \le  |x \beta|_{g\YY}$.
\endproof

Equation  \pref{e:7.2}, and hence Lemma~\ref{l:7.2},  now follows from the following two consequences of Lemma~\ref{elaborate}.
 
$$p |A|_\YY  = p|A|_{\YY,P} = p|P| =  |\group| $$
$$ q |x \beta|_\YY - q |xB|_\YY \ge q\sum_{g \in Q^x}  ( |x \beta|^*_{g\YY} -  |xB|^*_{g\YY})  = q |Q| \ = \ |\group| $$

\subsection{ \cite[Lemmas~7.1 and 7.2]{kv:equivariant}} \label{s:last}

Fix two ideal edges $\alpha$ and $\beta$ of $R$. Let $P, Q<\group$ denote the respective stabilizers of $\alpha$ and $\beta$. The indices of $P$ and $Q$ in $\group$ are respectively $p$ and $q$. Choose double coset representatives $x_1=1, x_2,\dots,x_k$ for $P\backslash\group\slash Q$, i.e.\ $$\group=PQ+Px_2Q+\dots+Px_kQ$$ where $+$ denotes disjoint union. The intersections $\gamma:=\alpha\cap\group\beta$ and $\gamma':=\beta\cap\group\alpha$ decompose as disjoint unions
\begin{align*}
\gamma&=\gamma_1+\dots+\gamma_k\\
\gamma'&=\gamma_1'+\dots+\gamma'_k
\end{align*}
where $\gamma_i:=\alpha\cap Px_i\beta$ and $\gamma':=\beta\cap Qx_i^{-1}\alpha$.

 We assume that at least one $\gamma_i$ is non-empty; i.e. $G\alpha$ and $G\beta$ {\it cross} in the notation of \cite{kv:equivariant}.  After replacing  $\beta$  by some translate $x\beta$, $x \in \group$,  we  assume that  $\gamma_1 = \alpha\cap\beta\not=\emptyset$. If $\gamma_i = \emptyset$ for all $i \ne 1$ then $\group\alpha$ and $\group_
 \beta$ {\it cross simply}.   Equivalently Notation~\ref{notn:intermediate}\pref{item:simple} is satisfied.
 
Recall  that we have enumerated the elements $K_1, K_2, \ldots$ of $\sgs$.  Let   ${\sfY_i}  $ be the  set of pieces for $K_i$ (see Section~\ref{s:Stallings}) and for   each $A \subset E(\Gamma)$, define 
$$|A|^*_{K_i}  := \sum_{Y \in \sfY_i} |A|^*_{\partial Y} \qquad     |A|_{K_i}:=\sum_{x\in \group} |A|^*_{xK_i}\qquad |A| :=(|A|_{K_1},\dots)\in\Z^\sgs$$

We can now state our versions of \cite[Lemmas~7.1 and 7.2]{kv:equivariant}.  They follow immediately from the definitions and Lemmas~\ref{lem:intermediate} and \ref{l:7.2}.

\begin{lemma}[cf.\ {\cite[Lemma~7.1]{kv:equivariant}}]\label{l:7.1}
If $\group\alpha$ and $\group\beta$ cross simply, and $P\le Q$, then 
$$p|\alpha\cap \beta|+q|\beta\cup Q\alpha|\le p|\alpha|+q|\beta|$$
\end{lemma}

\begin{lemma}[cf.\ {\cite[Lemma~7.2]{kv:equivariant}}]\label{l:7.2}  If $\group \alpha$ and $\group \beta$ cross, then for all $i$,
 $$ p|\alpha-\gamma_i|+q|\beta-\gamma_i'|\leq p|\alpha|+q|\beta|$$
 \end{lemma}
 
 \section{Extensions}\label{s:ext}
Let $\vec K=(K_1,K_2,\cdots,K_N)$ be a finite sequence of conjugacy
classes of finitely generated subgroups of $F_n$. Less formally, we
can (and will!) think of $K_i$ as a subgroup of $F_n$ defined up to
conjugation. Let $GMc(\vec K)$ be the associated generalized McCool
group, i.e. the subgroup of $Out(F_n)$ consisting of outer automorphisms
that preserve $K_i$ for all $i$; see Section~\ref{s:compare}. If $H$ is a subgroup of $F_n$ let
$N(H)=\{g\in F_n\mid gHg^{-1}=H\}$ be the normalizer of $H$ in $F_n$
and consider the natural homomorphism $N(H)\to Aut(H)$ that sends
$g\in N(H)$ to conjugation by $g$. When $H$ is noncyclic this
homomorphism is injective and we define
$$Out'(H)=Aut(H)/N(H)$$
When $H$ is cyclic we define $Out'(H)=Out(H)$, which is cyclic of
order 2 when $H$ is nontrivial. When $H$ is its own normalizer (i.e. $N(H)=H$)
then $Out'(H)=Out(H)$. 

There is then a well defined
restriction homomorphism
$$\Phi:GMc(\vec K)\to Out'(K_1)\times Out'(K_2)\times\cdots\times Out'(K_N)$$
Note that the kernel of $\Phi$ is the McCool group $Mc(\vec K)$ (see \cite{gl:mccool} and Section~\ref{s:compare}). In this section we are
interested in the image of $\Phi$, which we denote by $Ext(F_n;\vec
K)$. Thus an element of $Ext(F_n;\vec
K)$ is a tuple $(\theta_1,\cdots,\theta_N)$ of elements of $Out'(K_i)$
that {\it simultaneously extend} to an outer
automorphism of $F_n$. In \cite{gl:random} it was shown that when
$N=1$ and $K_1$ is a ``random'' subgroup of $F_n$ then $Ext(F_n;\vec
K)$ is trivial. The main theorem in this section is:

\begin{thm}\label{ext}
  $Ext(F_n;\vec K)$ is commensurable to a finite product of
  generalized McCool groups. In particular, it is of type $\vf$.
\end{thm}

Recall that two groups are commensurable if they contain isomorphic
finite index subgroups. If there is an epimorphism $G\to H$ with
finite kernel, and if $G$ is residually finite, then $G$ and $H$ are
commensurable. Finally, recall that $Out(F_n)$ (and hence any of its
subgroups) is residually finite \cite{grossman}.

We start by considering a special case.

\begin{lemma}\label{point}
  Suppose $F_n$ admits no free splitting or a splitting over a cyclic
  subgroup in which every $K_i$ is elliptic. Then the kernel of $\Phi$
  is finite and thus $Ext(F_n;\vec K)$ is commensurable with $GMc(\vec
  K)$.
\end{lemma}

\begin{proof}
   If the kernel is infinite, we have a
  sequence $f_j\in Out(F_n)$ of distinct automorphisms
  acting as the identity on
  each $K_i$. Then after a subsequence the limiting procedure
  \cite{paulin, mb} yields a nontrivial stable $F_n$-$\R$-tree $T$
  with each $K_i$ elliptic and with all arc stabilizers cyclic
  (in fact, $T$ belongs to the boundary of Outer space). But then by
  \cite{bf, french} $F_n$ has a cyclic splitting with each $K_i$
  elliptic, contradiction.
\end{proof}

The next case is that $F_n$ admits splittings over infinite cyclic
subgroups with all $K_i$ elliptic, but no such free splittings. In
that case we can consider the cyclic JSJ decomposition $J$ of $F_n$
\cite{rips-sela,dunwoody-sageev,papa-fu,gl:jsj}
relative to $\vec K$. Recall that $J$ is a cocompact simplicial $F_n$-tree in
which all $K_i$ are elliptic.

The JSJ decomposition $J$ has a lot of structure: there are two kinds
of vertices -- rigid and quadratically hanging (QH). The latter
correspond to surfaces with boundary, where boundary components
represent the incident edge groups or are commensurable with some
$K_i$. The surfaces may be nonorientable, but are required to contain
intersecting 2-sided simple closed curves, so e.g. a
pair of pants is not allowed. The key universal property that $J$
possesses is the following: if $g\in F_n$ is elliptic in $J$ and not
contained in a QH vertex group as a nonboundary element, then $g$ is
elliptic in every cyclic splitting of $F_n$ in which all $K_i$ are
elliptic. In particular, a rigid vertex group $V$ does not admit any
cyclic splittings with $\partial V$ and all $K_i$ contained in $V$
elliptic.

The JSJ tree is usually not unique up to
  isomorphism, but the set of elliptic elements is independent of the
  choice and
  the set of possible JSJ trees forms a contractible deformation space
  \cite{forester, clay, gl:deformation}. All edge stabilizers of $J$
  are cyclic, the set of noncyclic vertex stabilizers is independent
  of the choice of $J$ (maximal noncyclic subgroups of the set of
  elliptics), and so is the set of commensurability classes of edge
  stabilizers (intersections of noncyclic point stabilizers).

  In particular, for every noncyclic vertex group $V$,
  the set $\partial V$ of commensurability classes of edge groups
  contained in $V$ is well defined. In other words, if $f$ is an outer
  automorphism of $F_n$ that preserves $\vec K$, it may not fix $J$,
  but it preserves the conjugacy classes of noncyclic vertex groups
  (it is allowed to permute them)
  and commensurability classes of edge groups. It also preserves QH
  vertex groups and their peripheral subgroups.

  \begin{lemma}\label{jsj}  
    Suppose $F_n$ admits splittings over infinite cyclic
subgroups with all $K_i$ elliptic, but no such free splittings.
   Then $Ext(F_n;\vec K)$ is commensurable with the product
    of finitely many generalized McCool groups.
  \end{lemma}

  \begin{proof}
  There are two kinds of $K_i$, the ones that are commensurable to
  edge groups of $J$ or to boundary components of QH vertex groups,
  and the ones that are not -- therefore their representatives fix
  only vertices.

  To fix ideas, we now assume that all $K_i$ are of the latter
  kind. Then each $K_i$ has a unique noncyclic vertex group (up to
  conjugacy) $V(K_i)$ that contains it and these vertex groups are
  rigid. Noncyclic vertex groups are their own normalizer, and
  moreover, if $gKg^{-1}<V$ for some $K<V$ and $g\in F_n$ then either
  $g\in V$ or $K$ is contained in an edge group.  Thus $K_i$ gives a
  unique conjugacy class of subgroups of $V(K_i)$.  Let
  $V_1,V_2,\cdots,V_k$ be the representatives of vertex groups that
  occur in this way and let $\vec H_j$ be the tuple of associated
  subgroups $K_i$ that occur in $V_j$. We then claim that the group
  $Ext(F_n;\vec K)$ contains a subgroup of finite index isomorphic to
  $$Ext(V_1;\vec H'_1)\times\cdots\times Ext(V_k;\vec H'_k)$$ where
  $\vec H'_j$ is obtained from $\vec H_j$ by appending $V_j$-conjugacy
  classes of cyclic subgroups in $\partial V_j$, one representative of
  each commensurability class. Note that some groups in $\partial V_j$
  could be conjugate in $F_n$ but not in $V_j$; however, the number of
  conjugacy classes up to commensurability in $V_j$ we obtain in this
  way is still finite and does not exceed the number of $V_j$-orbits
  of edges incident to the vertex fixed by $V_j$. 
  Since $V_j$ does not admit a cyclic splitting
  with $\vec H_j$ and $\partial V_j$ elliptic, Lemma \ref{point}
  finishes the proof. To prove the claim, note that every simultaneous
  extension permutes the vertex groups, hence fixes the conjugacy
  classes of the $V_j$, and it permutes the conjugacy classes in
  $\partial V_j$. After taking a suitable finite index subgroup of
  $Ext(F_n;\vec K)$, there
  will be simultaneous extensions that fix $\partial V_j$ and all
  peripheral subgroups of QH vertex groups. Conversely,
  if automorphisms can be simultaneously extended to the $V_j$'s
  fixing $\partial V_j$, then they can be further extended by the
  identity to the rest of $F_n$.

  In the general case, when some $K_i$ are commensurable to edge
  groups or to boundary components of QH vertex groups, we ignore
  such $K_i$ and we reach the same conclusion as above since our
  extensions fix all edge groups of $J$ and all peripheral subgroups
  of QH vertex groups.
  \end{proof}

\begin{proof}[Proof of Theorem \ref{ext}]
  Let $\{G_1,\cdots,G_k\}$ be the
  smallest free factor system such that each $K_i$ is conjugate
  into some $G_j$. This $G_j$ will then be unique since free factors
  are malnormal. Note that if this free factor system is $\{F_n\}$
  then the conclusion follows from Lemma \ref{jsj}. 
  After reindexing, $\vec K$ yields $k$ tuples 
  $\vec H_1,\cdots,\vec H_k$ of conjugacy classes
  of subgroups of $G_1,\cdots,G_k$, and we observe that
  $$Ext(F_n;\vec K)=Ext(G_1;\vec H_1)\times \cdots\times Ext(G_k;\vec
  H_k)$$
  Indeed, if a tuple of automorphisms of each $\vec H_j$
  simultaneously extends to $G_j$, then these extensions can be
  combined to an automorphism of $G_1*\cdots *G_k$ which can then be
  further extended to $F_n$ since the free product is a free factor of
  $F_n$. Conversely, any automorphism of $F_n$ that preserves $\vec K$
  will also preserve each $G_j$ and the statement follows.

  By Lemma \ref{jsj} each $Ext(G_j;\vec H_j)$ is commensurable to a
  product of generalized McCool groups, so the same holds for
  $Ext(F_n;\vec K)$.
\end{proof}

\bibliographystyle{amsalpha}
\bibliography{../egersten.bfh/ref.bib}

\end{document}